\documentclass[12pt,a4paper]{amsart}
\usepackage{amssymb,eucal,amsxtra}
\usepackage{tikz-cd}
\usepackage{hyperref}

\calclayout

\newcommand{\+}{\nobreakdash-}
\renewcommand{\:}{\colon}
\renewcommand{\.}{\mskip.5\thinmuskip}

\newcommand{\ot}{\otimes}
\newcommand{\ocn}{\odot}
\newcommand{\rarrow}{\longrightarrow}

\newcommand{\oc}{\mathbin{\text{\smaller$\square$}}}
\newcommand{\suboc}{\mathbin{\text{\smaller$\.\scriptstyle\square$}}}
\newcommand{\bu}{{\text{\smaller\smaller$\scriptstyle\bullet$}}}
\newcommand{\lrarrow}{\.\relbar\joinrel\relbar\joinrel\rightarrow\.}
\newcommand{\twoheads}{\relbar\joinrel\twoheadrightarrow}

\newcommand{\vect}{{\operatorname{\mathsf{--vect}}}}
\newcommand{\modl}{{\operatorname{\mathsf{--mod}}}}
\newcommand{\comodl}{{\operatorname{\mathsf{--comod}}}}
\newcommand{\comodr}{{\operatorname{\mathsf{comod--}}}}
\newcommand{\bicomod}{{\operatorname{\mathsf{--comod--}}}}
\newcommand{\contra}{{\operatorname{\mathsf{--contra}}}}
\newcommand{\contrar}{{\operatorname{\mathsf{contra--}}}}

\DeclareMathOperator{\Hom}{Hom}
\DeclareMathOperator{\Ext}{Ext}

\DeclareMathOperator{\Cohom}{Cohom}
\DeclareMathOperator{\Cotor}{Cotor}
\DeclareMathOperator{\Coext}{Coext}
\DeclareMathOperator{\Ctrtor}{Ctrtor}

\newcommand{\Hot}{\mathsf{Hot}}
\newcommand{\Acycl}{\mathsf{Acycl}}

\newcommand{\A}{\mathcal A}
\newcommand{\B}{\mathcal B}
\newcommand{\C}{\mathcal C}
\newcommand{\D}{\mathcal D}
\newcommand{\E}{\mathcal E}
\newcommand{\I}{\mathcal I}
\newcommand{\J}{\mathcal J}
\newcommand{\K}{\mathcal K}
\renewcommand{\L}{\mathcal L}
\newcommand{\M}{\mathcal M}
\newcommand{\N}{\mathcal N}
\renewcommand{\S}{\mathcal S} 

\newcommand{\F}{\mathfrak F}
\newcommand{\fI}{\mathfrak I}
\newcommand{\fK}{\mathfrak K}
\renewcommand{\P}{\mathfrak P}
\newcommand{\fS}{\mathfrak S}
\newcommand{\fT}{\mathfrak T}

\newcommand{\sA}{\mathsf A}
\newcommand{\sB}{\mathsf B}
\newcommand{\sC}{\mathsf C}
\newcommand{\sD}{\mathsf D}
\newcommand{\sE}{\mathsf E}
\newcommand{\sF}{\mathsf F}
\newcommand{\sS}{\mathsf S}
\newcommand{\sP}{\mathsf P}
\newcommand{\sJ}{\mathsf J}

\newcommand{\sG}{\mathsf G}
\newcommand{\sH}{\mathsf H}
\newcommand{\sX}{\mathsf X}
\newcommand{\sY}{\mathsf Y}

\renewcommand{\b}{{\mathsf{b}}}
\newcommand{\co}{{\mathsf{co}}}
\newcommand{\ctr}{{\mathsf{ctr}}}
\newcommand{\abs}{{\mathsf{abs}}}

\newcommand{\inj}{{\mathsf{inj}}}
\newcommand{\proj}{{\mathsf{proj}}}

\renewcommand{\ss}{{\mathrm{ss}}}
\newcommand{\rop}{{\mathrm{op}}} 

\newcommand{\boR}{\mathbb R}
\newcommand{\boL}{\mathbb L}

\newcommand{\Section}[1]{\bigskip\section{#1}\medskip}

\theoremstyle{plain}
\newtheorem{thm}{Theorem}[section]

\newtheorem{lem}[thm]{Lemma}
\newtheorem{prop}[thm]{Proposition}
\newtheorem{cor}[thm]{Corollary}
\theoremstyle{definition}

\newtheorem{exs}[thm]{Examples}
\newtheorem{rem}[thm]{Remark}

\begin{document}

\title{Pseudo-dualizing complexes of bicomodules \\
and pairs of t-structures}
\author{Leonid Positselski}

\address{Institute of Mathematics of the Czech Academy of Sciences,
\v Zitn\'a~25, 115~67 Prague~1, Czech Republic; and
\newline\indent Laboratory of Algebra and Number Theory, Institute for
Information Transmission Problems, Moscow 127051, Russia}

\email{positselski@yandex.ru}

\begin{abstract}
 This paper is a coalgebra version of~\cite{Pps} and a sequel
to~\cite{Pmc}.
 We present the definition of a pseudo-dualizing complex of bicomodules
over a pair of coassociative coalgebras $\C$ and~$\D$.
 For any such complex $\L^\bu$, we construct a triangulated category
endowed with a pair of (possibly degenerate) t\+structures of
the derived type, whose hearts are the abelian categories of
left $\C$\+comodules and left $\D$\+contramodules.
 A weak version of pseudo-derived categories arising out of
(co)resolving subcategories in abelian/exact categories with enough
homotopy adjusted complexes is also considered.
 Quasi-finiteness conditions for coalgebras, comodules, and
contramodules are discussed as a preliminary material.
\end{abstract}

\maketitle

\tableofcontents

\section{Introduction}
\medskip

\subsection{{}} \label{introd-pseudo-derived}
 The philosophy of pseudo-derived equivalences was invented for
the purposes of $\infty$\+tilting theory in~\cite{PS2} and applied to 
pseudo-dualizing complexes of bimodules in~\cite{Pps}.
 The latter are otherwise known as ``semi-dualizing complexes'' in
the literature~\cite{Chr,HW}.

 In its full form, a pseudo-derived equivalence between two abelian 
categories means an equivalence between their exotic derived categories
standing in some sense ``in between'' the conventional unbounded derived
and co/contraderived categories.
 To be more precise, a ``pseudo-derived category'' means either
a pseudo-coderived or a pseudo-contraderived category.
 A construction of such pseudo-derived categories associated with 
(co)resolving subcategories closed under (co)products in abelian
categories with exact (co)products was suggested in~\cite{PS2} and
used in~\cite{Pps}.

 Let $\sA$ be an abelian category with exact coproduct functors.
 The \emph{coderived category} $\sD^\co(\sA)$ is defined as
the triangulated Verdier quotient category of the homotopy category
of unbounded complexes $\Hot(\sA)$ by its minimal full triangulated
subcategory $\Acycl^\co(\sA)$ containing the totalizations of short
exact sequences of complexes in $\sA$ and closed under coproducts.
 The condition of exactness of coproducts in $\sA$ guarantees that
all the complexes from $\Acycl^\co(\sA)$ are acyclic in the conventional
sense, $\Acycl^\co(\sA)\subset\Acycl(\sA)$; hence the conventional
derived category $\sD(\sA)$ is a Verdier quotient category of
the coderived category $\sD^\co(\sA)$.

 Let $\sE\subset\sA$ be a \emph{coresolving subcategory}, i.~e., a full
subcategory closed under extensions and cokernels of monomorphisms,
and such that every object of $\sA$ is a subobject of an object
from~$\sE$.
 Assume further that the full subcategory $\sE$ is closed under
coproducts in~$\sA$.
 As explained in~\cite[Section~4]{PS2} and~\cite[Section~1]{Pps},
the derived category $\sD(\sE)$ of the exact category $\sE$ is
an intermediate Verdier quotient category between $\sD^\co(\sA)$ and
$\sD(\sA)$; so the Verdier quotient functor $\sD^\co(\sA)\twoheads
\sD(\sA)$ is the composition of two triangulated Verdier quotient
functors $\sD^\co(\sA)\twoheads\sD(\sE)\twoheads\sD(\sA)$.
 The category $\sD(\sE)$ is called the \emph{pseudo-coderived category}
of $\sA$ associated with the coresolving subcategory $\sE\subset\sA$
closed under coproducts.

 Dually, let $\sB$ be an abelian category with exact products.
 The \emph{contraderived category} $\sD^\ctr(\sB)$ is the triangulated
Verdier quotient category of the homotopy category $\Hot(\sB)$ by
its minimal full triangulated subcategory containing the totalizations
of short exact sequences in $\sB$ and closed under products.
 The conventional derived category $\sD(\sB)$ is naturally a Verdier
quotient category of the contraderived category.

 Let $\sF\subset\sB$ be a \emph{resolving subcategory}, i.~e., a full
subcategory closed under extensions and kernels of monomorphisms
such that every object of $\sB$ is a quotient object of an object
from~$\sF$.
 Assume that $\sF$ is closed under products in~$\sB$.
 Then the derived category $\sD(\sF)$ of the exact category $\sF$ is
an intermediate Verdier quotient category between the contraderived
category $\sD^\ctr(\sB)$ and the derived category $\sD(\sB)$; so there
are triangulated Verdier quotient functors
$\sD^\ctr(\sB)\twoheads\sD(\sF)\twoheads\sD(\sB)$.
 The triangulated category $\sD(\sF)$ is called
the \emph{pseudo-contraderived category} of $\sB$ associated with
the resolving subcategory $\sF\subset\sB$ closed under products.

 A \emph{pseudo-derived equivalence} between abelian categories $\sA$
and $\sB$ is a triangulated equivalence $\sD(\sE)\simeq\sD(\sF)$
between their pseudo-derived categories associated with some full
subcategories $\sE\subset\sA$ and $\sF\subset\sB$.
 Following the papers~\cite{PS2,Pps}, pseudo-derived equivalences occur
in the $\infty$\+tilting theory and in connection with pseudo-dualizing
complexes of bimodules over a pair of associative rings.

\subsection{{}} \label{introd-t-structures}
 It appears, though, that coresolving subcategories closed under
coproducts and resolving subcategories closed under products are not as
common as one would like them to be.
 For example, if $\sA$ is an abelian category with enough injective
objects, then the full subcategory of injective objects $\sA_\inj$
is coresolving in $\sA$, but it is rarely closed under coproducts.
 Similarly, if $\sB$ is an abelian category with enough projectives,
then the full subcategory of projective objects $\sB_\proj\subset\sB$ is
resolving in $\sB$, but it is rarely closed under products.

 What can one say about the derived category $\sD(\sE)$ or $\sD(\sF)$ of
a (co)resolving subcategory $\sE\subset\sA$ or $\sF\subset\sB$ that is
not closed under (co)products?
 How does one interpret a triangulated equivalence $\sD(\sE)\simeq
\sD(\sF)$ in terms of the original abelian categories $\sA$ and~$\sB$\,?
 For this purpose, a weak version of pseudo-derived equivalences with
an equivalence of pseudo-derived categories replaced by a pair of
t\+structures on a single triangulated category was discussed
in~\cite[Section 5]{PS2}.

 Specifically, let $\sA$ be an abelian category and $\sE\subset\sA$ be
a coresolving subcategory.
 Denote by $\sD^{\ge0}(\sE)$ the full subcategory in $\sD(\sE)$
consisting of all the objects that can be represented by nonnegatively
cohomologically graded complexes in~$\sE$.
 Furthermore, denote by $\sD_\sA^{\le0}(\sE)\subset\sD(\sE)$ the full
subcategory of all objects $E^\bu$ whose images in $\sD(\sA)$ have their
cohomology objects $H_\sA^n(E^\bu)\in\sA$ concentrated in nonpositive
cohomological degrees~$n$.
 Then the pair of full subcategories $\sD_\sA^{\le0}(\sE)$ and
$\sD^{\ge0}(\sE)\subset\sD(\sE)$ is a t\+structure on $\sD(\sE)$
with the heart naturally equivalent to~$\sA$.
 Moreover, it is a t\+structure \emph{of the derived type}, i.~e.,
the Ext groups between objects of the heart computed in the triangulated
category $\sD(\sE)$ agree with the Ext groups in the abelian
category~$\sA$ \cite[Proposition~5.5]{PS2}.
 
 Dually, if $\sB$ is an abelian category and $\sF\subset\sB$ is
a resolving subcategory, then the pair of full subcategories
$\sD^{\le0}(\sF)$ and $\sD^{\ge0}_\sB(\sF)$ is a t\+structure of
the derived type on $\sD(\sF)$ with the heart naturally equivalent
to~$\sB$.
 Such t\+structures can well be \emph{degenerate}, though:
the intersections $\bigcap_{n\ge0}\sD^{\ge n}(\sE)$ and
$\bigcap_{n\le0}\sD^{\le n}(\sF)$ always vanish, but the intersections
$\bigcap_{n\le0}\sD_\sA^{\le n}(\sE)$ and
$\bigcap_{n\ge0}\sD_\sB^{\ge n}(\sF)$ are often
nontrivial~\cite[Remark~5.6]{PS2}.

 Thus, if the triangulated categories $\sD(\sE)$ and $\sD(\sF)$ happen
to be equivalent, we obtain a pair of t\+structures of the derived
type on one and the same triangulated category $\sD(\sE)=\sD=\sD(\sF)$.
 The hearts of these two t\+structures are the abelian categories $\sA$
and $\sB$, respectively.
 Under a natural additional assumption of nontriviality, we call such
a situation a \emph{t\+derived pseudo-equivalence} between the abelian
categories $\sA$ and~$\sB$.
 The aim of this paper is to show that such a situation does occur in
connection with what we call a \emph{pseudo-dualizing complex of
bicomodules} over a pair of coalgebras $\C$ and~$\D$.

\subsection{{}} \label{introd-homotopy-adjusted}
 We also suggest an alternative point of view on pseudo-derived
categories, which is in some way intermediate between the approaches
of Section~\ref{introd-pseudo-derived} (i.~e., \cite[Section~4]{PS2})
and Section~\ref{introd-t-structures} (i.~e., \cite[Section~5]{PS2}).

 Namely, let $\sA$ be an exact category.
 We say that $\sA$ \emph{has enough homotopy injective complexes of
injectives} if every object of the derived category $\sD(\sA)$ can
be represented by a homotopy injective complex of injective objects
in $\sA$, i.~e., a complex of injectives that is right orthogonal
to all acyclic complexes in the homotopy category $\Hot(\sA)$.
 For example, the abelian category of $\C$\+comodules has enough
homotopy injective complexes of injectives for any coalgebra~$\C$
\cite[Theorem~2.4(a)]{Pkoszul}, \cite[Theorem~1.1(c)]{Pmc};
moreover, any Grothendieck abelian category has enough homotopy
injective complexes of injectives~\cite{AJS,Ser,Gil}.

 Let $\sE\subset\sA$ be a coresolving subcategory containing
the injective objects.
 Then the functor $\sD(\sE)\rarrow\sD(\sA)$ induced by the inclusion
$\sE\rarrow\sA$ is a Verdier quotient functor, and it also has
a (fully faithful) right adjoint.
 Moreover, there is the following diagram of triangulated functors:
\begin{equation} \label{homotopy-injectives-coresolving}
\begin{tikzcd}
\Hot(\sA_\inj) \arrow[d] \arrow[dd, two heads, bend right = 50] \\
\sD(\sE) \arrow[d, two heads] \\
\sD(\sA) \arrow[u, tail, bend right = 45]
\arrow[uu, tail, bend right = 65]
\end{tikzcd}
\end{equation}
 Here arrows with a tail denote fully faithful functors and arrows with
two heads denote Verdier quotient functors.
 The downwards directed curvilinear arrow is the composition of the two
straight downwards directed arrows, while the upwards directed
curvilinear arrows denote right adjoint functors to the downwards
directed arrows.
 The functor $\sD(\sA)\rarrow\sD(\sE)$ is the composition $\sD(\sA)
\rarrow\Hot(\sA_\inj)\rarrow\sD(\sE)$.

 Dually, let $\sB$ be an exact category with enough homotopy projective
complexes of projectives.
 For example, the abelian category of $\C$\+contramodules has enough
homotopy projective complexes of projectives for any coalgebra~$\C$
\cite[Theorem~2.4(b)]{Pkoszul}, \cite[Theorem~1.1(a)]{Pmc};
moreover, any locally presentable abelian category with enough
projective objects has enough homotopy projective complexes of
projectives~\cite[Corollary~6.7]{PS4}.
 Let $\sF\subset\sB$ be a resolving subcategory containing
the projective objects.
 Then the functor $\sD(\sF)\rarrow\sD(\sB)$ induced by the inclusion
$\sF\rarrow\sB$ is a Verdier quotient functor, and it also has
a (fully faithful) left adjoint.
 Moreover, there is the following diagram of triangulated functors:
\begin{equation} \label{homotopy-projectives-resolving}
\begin{tikzcd}
\Hot(\sB_\proj) \arrow[d] \arrow[dd, two heads, bend left = 50] \\
\sD(\sF) \arrow[d, two heads] \\
\sD(\sB) \arrow[u, tail, bend left = 45]
\arrow[uu, tail, bend left = 65]
\end{tikzcd}
\end{equation}
 Here the downwards directed curvilinear arrow is the composition of
the two straight downwards directed arrows, while the upwards directed
curvilinear arrows denote left adjoint functors to the downwards
directed arrows.
 The functor $\sD(\sB)\rarrow\sD(\sF)$ is the composition $\sD(\sB)
\rarrow\Hot(\sB_\proj)\rarrow\sD(\sF)$.

\subsection{{}} \label{introd-dedualizing-bicomod}
 The philosophy of \emph{dualizing} and \emph{dedualizing complexes} was
discussed in the paper~\cite{Pmgm} and, in the context of coalgebras,
in the paper~\cite{Pmc}.
 Briefly put, dualizing complexes induce equivalences between
the coderived and contraderived categories, while dedualizing complexes
provide equivalences between conventional derived categories.
 In particular, given an associative ring $A$, the one-term complex of
$A$\+$A$\+bimodules $A$ is the simplest example of a dedualizing complex
of bimodules, related to the identity derived equivalence $\sD(A\modl)=
\sD(A\modl)$; while the datum of a dualizing complex $D^\bu$ for a pair
of rings $A$ and $B$ leads to a triangulated equivalence
$\sD^\co(A\modl)\simeq\sD^\ctr(B\modl)$ \cite{Pfp,Pps}.

 Let $k$~be a fixed ground field and $\C$ be a (coassociative, counital)
coalgebra over~$k$.
 Then there are two kinds of abelian categories that one can assign
to~$\C$: in addition to the more familiar categories of left and right
$\C$\+comodules $\C\comodl$ and $\comodr\C$, there are also less
familiar, but no less natural abelian categories of left and right
\emph{$\C$\+contramodules} $\C\contra$ and $\contrar\C$ \cite{Prev}.

 The abelian category of $\C$\+comodules has exact functors of infinite
coproducts and enough injective objects, while the abelian category of
$\C$\+contramodules has exact functors of infinite products and enough
projective objects.
 There is a fundamental homological phenomenon of
\emph{comodule-contramodule correspondence}, meaning a natural
triangulated equivalence between the coderived category of comodules
and the contraderived category of
contramodules~\cite[Sections~0.2.6\+-0.2.7]{Psemi},
and~\cite[Sections~4.4 and~5.2]{Pkoszul}
\begin{equation} \label{derived-co-contra}
 \sD^\co(\C\comodl)\simeq\Hot(\C\comodl_\inj)\simeq
 \Hot(\C\contra_\proj)\simeq\sD^\ctr(\C\contra).
\end{equation}

 The triangulated equivalence~\eqref{derived-co-contra} is induced by
an equivalence between the additive categories of injective left
$\C$\+comodules and projective left $\C$\+contramodules,
$\C\comodl_\inj\simeq\C\contra_\proj$.
 The latter equivalence is provided by the adjoint functors of left
$\C$\+comodule homomorphisms from and the \emph{contratensor product}
with the $\C$\+$\C$\+bicomodule~$\C$,
\begin{equation} \label{underived-co-contra}
 \Hom_\C(\C,{-})\:\C\comodl_\inj\simeq
 \C\contra_\proj\,:\!\C\ocn_\C{-}.
\end{equation}
 Thus the one-term complex of $\C$\+$\C$\+bicomodules $\C$ is
the simplest example of a dualizing complex of bicomodules.

 Let $\C$ and $\D$ be two coalgebras over the same ground field~$k$.
 The definition of a \emph{dedualizing complex} of
$\C$\+$\D$\+bicomodules was given in the paper~\cite[Section~3]{Pmc}.
 According to~\cite[Theorem~3.6]{Pmc}, the datum of a dedualizing
complex of bicomodules $\B^\bu$ for a left cocoherent coalgebra $\C$
and a right cocoherent coalgebra $\D$ over a field~$k$ induces
an equivalence of derived categories
\begin{equation} \label{dedualizing-bicomodule-equivalences}
\sD^\star(\C\comodl)\simeq\sD^\star(\D\contra)
\end{equation}
for any conventional derived category symbol $\star=\b$, $+$, $-$,
or~$\varnothing$.
 Moreover, similar triangulated equivalences can be constructed for any
\emph{absolute} derived category symbol $\star=\abs+$, $\abs-$,
or~$\abs$.
 The triangulated
equivalences~\eqref{dedualizing-bicomodule-equivalences} are provided
by the mutually inverse right derived functor of comodule homomorphisms
$\boR\Hom_\C(\B^\bu,{-})$ and left derived functor of
contratensor product $\B^\bu\ocn_\D^\boL{-}$.

\subsection{{}} \label{introd-finiteness}
 All the definitions of a dualizing complex of
(bi)modules~\cite{Har,Yek,YZ,Miy,CFH,Pfp} involve three kinds of
conditions: (i)~finite injective dimension, (ii)~finite generatedness,
and (iii)~homothety isomorphisms.
 Analogously, the definitions of a dedualizing complex
in~\cite{Pmgm,Pmc} involve (roughly) three kinds of conditions:
(i)~finite projective dimension, (ii)~finite (co)generatedness,
and (iii)~homothety isomorphisms.

 The definition of a pseudo-dualizing complex~\cite{Pps} (or, in
the more traditional terminology, a ``semi-dualizing
complex''~\cite{Chr,HW}) is obtained from that of a (de)dualizing
complex by dropping the finite injective/projective dimension
condition~(i), while retaining the finite generatedness and homothety
isomorphism conditions~(ii\+-iii).

 In all these situations, the (derived) homothety isomorphism
conditions are rather straightforward to formulate, but the finite
(co)generatedness conditions are complicated, with many alternative
versions of them considered in various papers.
 In particular, the paper~\cite{Pmc} starts with a discussion of
finite (co)generatedness and (co)presentability conditions in comodule
and contramodule categories in~\cite[Section~2]{Pmc}.
 As we mentioned in Section~\ref{introd-dedualizing-bicomod},
the construction of the derived equivalence in~\cite[Theorem~3.6]{Pmc}
(see~\eqref{dedualizing-bicomodule-equivalences} above) was given in
the assumption of cocoherence conditions on the coalgebras $\C$ and~$\D$.

 Finite cogeneratedness and finite copresentability conditions are
not very natural for coalgebras, though, as they are not invariant
under the Morita--Takeuchi equivalences of coalgebras.
 Quasi-finite cogeneratedness and quasi-finite copresentability
conditions, going back to Takeuchi's classical paper~\cite{Tak},
are generally preferable.
 So this paper starts with a discussion of quasi-finitely cogenerated
and copresented comodules and quasi-finitely generated and presented
contramodules in Section~\ref{quasi-finite-secn}.

 We strived to relax the finite generatedness/presentability conditions
as much as possible in the paper~\cite{Pps}, and we do likewise in
the present paper.
 The result is that no coherence assumptions about the rings $A$ and $B$
are used in the main results of~\cite{Pps}, and no cocoherence or
quasi-cocoherence assumptions about the coalgebras $\C$ and $\D$ are
made in the main results of the present paper.
 Instead, we imposed appropriate finite presentability conditions
on the pseudo-dualizing complex of bimodules $L^\bu$ in
the paper~\cite{Pps}, and we impose appropriate quasi-finite
copresentability conditions on the pseudo-dualizing complex of
bicomodules $\L^\bu$ in this paper.

\subsection{{}} \label{introd-main-results}
 Let $\C$ and $\D$ be coassociative coalgebras over a fixed field~$k$.
 A \emph{pseudo-dualizing complex} $\L^\bu$ for the coalgebras $\C$
and $\D$ (cf.\ the discussion of ``semidualizing bicomodules''
in~\cite{GLT}) is a finite complex of $\C$\+$\D$\+bicomodules satisfying
the following two conditions:
\begin{enumerate}
\renewcommand{\theenumi}{\roman{enumi}}
\setcounter{enumi}{1}
\item as a complex of left $\C$\+comodules, $\L^\bu$ is quasi-isomorphic
to a bounded below complex of quasi-finitely cogenerated injective
$\C$\+comodules, and similarly, as a complex of right $\D$\+comodules,
$\L^\bu$ is quasi-isomorphic to a bounded below complex of
quasi-finitely cogenerated injective $\D$\+comodules;
\item the homothety maps
$\C^*\rarrow\Hom_{\sD^\b(\comodr\D)}(\L^\bu,\L^\bu[*])$ and
$\D^*{}^\rop\rarrow\Hom_{\sD^\b(\C\comodl)}(\L^\bu,\L^\bu[*])$
are isomorphisms of graded rings. \hbadness=1700
\end{enumerate}
 This definition is obtained by dropping the finite projective and
contraflat dimension condition~(i) from the definition of
a \emph{dedualizing} complex of $\C$\+$\D$\+bicomodules $\B^\bu$
in~\cite[Section~3]{Pmc}, removing the cocoherence conditions on
the coalgebras, and rewriting the finite copresentability
condition~(ii) accordingly.
 Here the quasi-finite cogeneratedness is a natural weakening
of the finite cogeneratedness condition on comodules, having
the advantage of being Morita-invariant~\cite{Tak}, as
discussed above in Section~\ref{introd-finiteness}.

 The main result of this paper provides the following diagram of
triangulated functors associated with a pseudo-dualizing complex {of
$\C$\+$\D$\+bicomodules~$\L^\bu$ (cf.~\cite{Pps}): \hfuzz=6.5pt
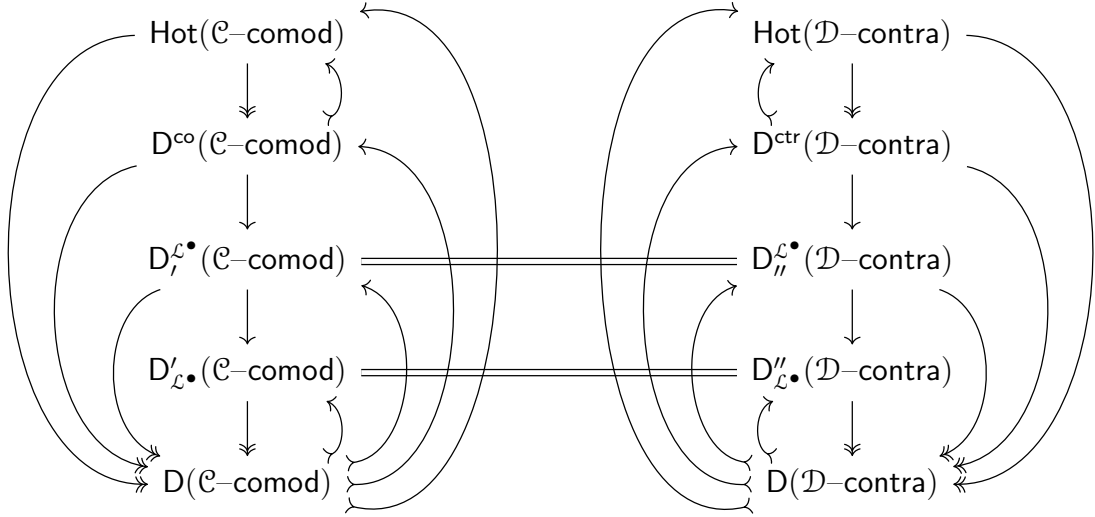
\begin{equation} \label{main-results-diagram}
\!\!\!\begin{tikzcd}
\Hot(\C\comodl) \arrow[d, two heads]
\arrow[dddd, two heads, bend right=90] &&&&&
\Hot(\D\contra) \arrow[d, two heads]
\arrow[dddd, two heads, bend left=90] \\
\sD^\co(\C\comodl) \arrow[d] 
\arrow[u, tail, bend right=70]
\arrow[ddd, two heads, bend right=80] &&&&&
\sD^\ctr(\D\contra) \arrow[d] 
\arrow[u, tail, bend left=70]
\arrow[ddd, two heads, bend left=80] \\
\sD^{\L^\bu}_{\prime}(\C\comodl) \arrow[d]
\arrow[rrrrr, Leftrightarrow, no head, no tail]
\arrow[dd, two heads, bend right=72] &&&&&
\sD^{\L^\bu}_{\prime\prime}(\D\contra) \arrow[d]
\arrow[dd, two heads, bend left=72] \\
\sD'_{\L^\bu}(\C\comodl) \arrow[d, two heads]
\arrow[rrrrr, Leftrightarrow, no head, no tail] &&&&&
\sD''_{\L^\bu}(\D\contra) \arrow[d, two heads] \\
\sD(\C\comodl) \arrow[uuuu, tail, bend right=102]
\arrow[uuu, tail, bend right=90]
\arrow[uu, tail, bend right=78]
\arrow[u, tail, bend right=70] &&&&&
\sD(\D\contra) \arrow[uuuu, tail, bend left=102]
\arrow[uuu, tail, bend left=90]
\arrow[uu, tail, bend left=78]
\arrow[u, tail, bend left=70]
\end{tikzcd}
\end{equation}
 Here} the functors shown by arrows with two heads are Verdier quotient
functors, the functors shown by arrows with a tail are fully faithful,
and double lines show triangulated equivalences.
 The outer, downwards directed curvilinear arrows (with two heads) are
the compositions of the vertical straight arrows.
 The inner, upwards directed curvilinear arrows (with a tail) are
adjoint on the right (in the comodule part of the diagram) and on
the left (in the contramodule part of the diagram) to the compositions
of the vertical straight arrows.

 In particular, when $\L^\bu=\B^\bu$ is a dedualizing complex for
a pair of coalgebras $\C$ and $\D$, i.~e., the finite
projective/contraflat dimension condition~(i)
of~\cite[Section~3]{Pmc} is satisfied, one has
$\sD'_{\L^\bu}(\C\comodl)=\sD(\C\comodl)$ and $\sD''_{\L^\bu}(\D\contra)
=\sD(\D\contra)$.
 In other words, the lower two vertical arrows in
the diagram~\eqref{main-results-diagram}
are isomorphisms of triangulated categories.
 The lower triangulated equivalence in
the diagram~\eqref{main-results-diagram} coincides with the one
provided by~\cite[Theorem~3.6]{Pmc} in this case.

 When $\L^\bu=\C=\D$, one has $\sD_\prime^{\L^\bu}(\C\comodl)=
\sD^\co(\C\comodl)$ and $\sD_{\prime\prime}^{\L^\bu}(\D\contra)=
\sD^\ctr(\D\contra)$, that is the next-to-upper two vertical arrows
in the diagram~\eqref{main-results-diagram}
are isomorphisms of triangulated categories.
 The upper triangulated equivalence in 
the diagram~\eqref{main-results-diagram} is the derived
comodule-contramodule correspondence~\eqref{derived-co-contra}
in this case.
 More generally, the upper triangulated equivalence in
the diagram~\eqref{main-results-diagram} corresponding to
a \emph{dualizing complex} $\L^\bu=\K^\bu$ for a pair of coalgebras
$\C$ and $\D$ can be thought of as a part of derived
Morita--Takeuchi equivalence (cf.~\cite{Far}).

\subsection{{}}
 Among the five pairs of triangulated categories on
the diagram~\eqref{main-results-diagram}, there are two pairs which
depend on the pseudo-dualizing complex~$\L^\bu$.
 These four triangulated categories $\sD^{\L^\bu}_{\prime}(\C\comodl)$,
\,$\sD^{\L^\bu}_{\prime\prime}(\D\contra)$,
\,$\sD'_{\L^\bu}(\C\comodl)$, and $\sD''_{\L^\bu}(\D\contra)$
are constructed in the following way.

 Suppose that the finite complex $\L^\bu$ is situated in
the cohomological degrees $-d_1\le m\le d_2$.
 Then there are two pairs of sequences of full subcategories
$$
 \dotsb\subset\sE^{d_2+2}\subset\sE^{d_2+1}\subset\sE^{d_2}
 \subset\sE_{d_1}\subset\sE_{d_1+1}\subset\sE_{d_1+2}\subset\dotsb
 \subset\sA
$$
and
$$
 \dotsb\subset\sF^{d_2+2}\subset\sF^{d_2+1}\subset\sF^{d_2}
 \subset\sF_{d_1}\subset\sF_{d_1+1}\subset\sF_{d_1+2}\subset\dotsb
 \subset\sB
$$
in the abelian categories $\sA=\C\comodl$ and $\sB=\D\contra$.
 The subcategories with lower indices form increasing sequences,
while the subcategories with upper indices form decreasing sequences.
 The full subcategories $\sE_{l_1}$ and $\sE^{l_2}\subset\sA$,
where $l_1\ge d_1$ and $l_2\ge d_2$, are coresolving, while
the full subcategories $\sF_{l_1}$ and $\sF^{l_2}\subset\sB$
are resolving.

 For any two integers $l''_1\ge l'_1\ge d_1$, the category $\sE_{l_1''}$
has finite $\sE_{l_1'}$\+coresolution dimension, while the category
$\sF_{l_1''}$ has finite $\sF_{l_1'}$\+resolution dimension.
 Therefore, the derived category $\sD(\sE_{l_1})$ of the exact
category $\sE_{l_1}$ does not depend on the choice of an integer
$l_1\ge d_1$, and similarly, the derived category $\sD(\sF_{l_1})$ of
the exact category $\sF_{l_1}$ does not depend on the choice of~$l_1$.
 We set
$$
 \sD'_{\L^\bu}(\C\comodl)=\sD(\sE_{l_1})
 \quad\text{and}\quad
 \sD''_{\L^\bu}(\D\contra)=\sD(\sF_{l_1}).
$$

 For any two integers $l''_2\ge l'_2\ge d_2$, the category $\sE^{l_2'}$
has finite $\sE^{l_2''}$\+coresolution dimension, while the category
$\sF^{l_2'}$ has finite $\sF^{l_2''}$\+resolution dimension.
 Therefore, the derived category $\sD(\sE^{l_2})$ does not depend on
the choice of an integer $l_2\ge d_2$, and similarly, the derived
category $\sD(\sF^{l_2})$ does not depend on the choice of~$l_2$.
 We set
$$
 \sD_{\prime}^{\L^\bu}(\C\comodl)=\sD(\sE^{l_2})
 \quad\text{and}\quad
 \sD_{\prime\prime}^{\L^\bu}(\D\contra)=\sD(\sF^{l_2}).
$$

 Let us now briefly explain where the full subcategories
$\sE_{l_1}$, \,$\sE^{l_2}$, \,$\sF_{l_1}$, and $\sF^{l_2}$ come from.
 The full subcategories $\sE_{l_1}\subset\C\comodl$ and
$\sF_{l_1}\subset\D\contra$ are our analogues of what are
known as the \emph{Auslander} and \emph{Bass classes} in
the literature~\cite{Chr,FJ,CFH,HW,GLT}.
 So they are defined as the classes of all left $\C$\+comodules and
left $\D$\+contramodules satisfying certain conditions with respect to
the derived functors $\boR\Hom_\C(\L^\bu,{-})$ and
$\L^\bu\ocn_\D^\boL{-}$, with the parameter~$l_1$ meaning a certain
(co)homological degree.
 The full subcategory $\sF_{l_1}\subset\D\contra$ is an analogue of
the Auslander class and the full subcategory $\sE_{l_1}\subset\C\comodl$
is a version of the Bass class.
 These are the \emph{maximal corresponding classes} of objects in
the categories $\sA=\C\comodl$ and $\sB=\D\contra$ with respect to
the covariant duality defined by the pseudo-dualizing complex~$\L^\bu$
(cf.\ the discussions of the Auslander and Bass classes
in~\cite[Sections~0.7 and~3]{Pps} and the maximal $\infty$\+tilting
and $\infty$\+cotilting classes in~\cite[Sections~2\+-3]{PS2}).

 The full subcategories $\sE^{l^2}\subset\C\comodl$ and
$\sF^{l^2}\subset\D\contra$ are the \emph{minimal corresponding classes}
in the categories $\sA=\C\comodl$ and $\sB=\D\contra$.
 They are defined by a certain iterative generation procedure, starting
from the full subcategory of injectives in $\sA$ and the full
subcategory of projectives in $\sB$, and proceeding using the derived
functors $\boR\Hom_\C(\L^\bu,{-})$ and $\L^\bu\ocn_\D^\boL{-}$, with
the parameter~$l_2$, once again, meaning a certain
(co)homological degree.
 These are the analogues of the minimal corresponding classes
from~\cite[Sections~0.7 and~5]{Pps} and of the minimal $\infty$\+tilting
and $\infty$\+cotilting classes from~\cite[Lemma~3.6
and Example~3.7]{PS2}.

 The derived equivalences
$$
 \sD^\star(\sE_{l_1})\simeq\sD^\star(\sF_{l_1})
 \quad\text{and}\quad
 \sD^\star(\sE^{l_2})\simeq\sD^\star(\sF^{l_2})
$$
(including, in particular, the triangulated equivalences in
the diagram~\eqref{main-results-diagram} arising as the particular
cases for $\star=\varnothing$) are provided by \emph{appropriately
constructed} derived functors of comodule homomorphisms and contratensor
product $\boR\Hom_\C(\L^\bu,{-})$ and $\L^\bu\ocn_\D^\boL{-}$.
 The rather complicated constructions of these derived functors are
based on the technique developed in~\cite[Appendix~A]{Pps}.

 Following the discussion in Section~\ref{introd-t-structures},
each of the triangulated categories
$$
 \sD'_{\L^\bu}(\C\comodl)\simeq\sD''_{\L^\bu}(\D\contra)
 \quad\text{and}\quad
 \sD_{\prime}^{\L^\bu}(\C\comodl)\simeq
 \sD_{\prime\prime}^{\L^\bu}(\D\contra)
$$
carries two (very possibly degenerate) t\+structures of the derived type,
whose abelian hearts are the categories of left $\C$\+comodules
and left $\D$\+contramodules $\sA=\C\comodl$ and $\sB=\D\contra$.

\medskip
\textbf{Acknowledgement.} 
 I~am grateful to Jan \v St\!'ov\'\i\v cek for helpful discussions.
 The author's research is supported by the GA\v CR project 20-13778S
and research plan RVO:~67985840.

\Section{Homotopy Adjusted Complexes and (Co)Resolving Subcategories}
\label{homotopy-adjusted-coresolving-secn}

 In this section we prove the results promised in
Section~\ref{introd-homotopy-adjusted} of the introduction.

 Following the terminology in~\cite{Pps} (originating from
the assumptions of~\cite[Sections~A.3 and~A.5]{Pcosh}), which differs
slightly from the standard terminology, we do \emph{not} include
closedness under direct summands into the definition of a (co)resolving
subcategory (see Section~\ref{introd-pseudo-derived}).
 This allows us to avoid the unnecessary closure under direct summands
in the conditions~(I\+-IV) of Section~\ref{abstract-classes-secn} below.
 That is why we sometimes need to assume separately that a coresolving
subcategory contains the injective objects, or that a resolving
subcategory contains the projective objects.

 Let $\sA$ be an exact category (the reader will loose little by
assuming that $\sA$ is abelian).
 An (unbounded) complex $J^\bu$ in $\sA$ is said to be \emph{homotopy
injective} if for any acyclic complex $X^\bu$ in $\sA$ the complex of
abelian groups $\Hom_\sA(X^\bu,J^\bu)$ is acyclic.
 We say that an exact category $\sA$ has \emph{enough homotopy
injective complexes} if for any complex $M^\bu$ in $\sA$ there exists
a homotopy injective complex $J^\bu$ together with a quasi-isomorphism
$M^\bu\rarrow J^\bu$ of complexes in~$\sA$.

 Consider the canonical Verdier quotient functor from the homotopy
category to the derived category of unbounded complexes,
\begin{equation} \label{homotopy-derived-A}
 \Hot(\sA)\lrarrow\sD(\sA).
\end{equation}
 An exact category $\sA$ has enough homotopy injective complexes if and
only if the functor~\eqref{homotopy-derived-A} has a right adjoint.
 If this is the case, such a right adjoint functor assigns to
any complex $M^\bu$ in $\sA$ its homotopy injective resolution, i.~e.,
a homotopy injective complex $J^\bu$ endowed with a quasi-isomorphism
$M^\bu\rarrow J^\bu$.

 Moreover, we will say that an exact category $\sA$ has \emph{enough
homotopy injective complexes of injectives} if for any complex $M^\bu$
in $\sA$ there exists a homotopy injective complex \emph{of injective
objects} $J^\bu$ together with a quasi-isomorphism $M^\bu\rarrow J^\bu$
of complexes in $\sA$.
 If this is the case, then the right adjoint functor
$\sD(\sA)\rarrow\Hot(\sA)$ to the functor~\eqref{homotopy-derived-A}
factorizes through the homotopy category of complexes of injective
objects, $\sD(\sA)\rarrow\Hot(\sA_\inj)\rarrow\Hot(\sA)$.
 We will denote the resulting functor by $\theta\:\sD(\sA)\rarrow
\Hot(\sA_\inj)$.

 For example, any Grothendieck abelian category $\sA$ has enough
homotopy injective complexes~\cite[Theorem~5.4]{AJS}; moreover,
it has enough homotopy injective complexes of injective
objects~\cite[Theorem~3.13 and Lemma~3.7(ii)]{Ser},
\cite[Corollary~7.1]{Gil}.
 In particular, for any coassociative coalgebra $\C$ over a field~$k$,
the abelian category of left $\C$\+comodules $\sA=\C\comodl$ has
enough homotopy injective complexes of injectives~\cite[Sections~2.4
and~5.5]{Pkoszul} (see also~\cite[Theorem~1.1(c)]{Pmc}).

 Dually, let $\sB$ be an exact category (which is also going to be
abelian in most applications).
 An (unbounded) complex $P^\bu$ in $\sB$ is said to be \emph{homotopy
projective} if for any acyclic complex $Y^\bu$ in $\sB$ the complex
of abelian groups $\Hom_\sB(P^\bu,Y^\bu)$ is acyclic.
 We say that an exact category $\sB$ has \emph{enough homotopy
projective complexes} if for any complex $T^\bu$ in $\sB$ there exists
a homotopy projective complex $P^\bu$ together with a quasi-isomorphism
$P^\bu\rarrow T^\bu$ of complexes in~$\sB$.

 An exact category $\sB$ has enough homotopy projective complexes if
and only if the canonical Verdier quotient functor $\Hot(\sB)\rarrow
\sD(\sB)$ has a left adjoint (which then assigns to a complex $T^\bu$
in $\sB$ its homotopy projective resolution).

 Moreover, we will say that an exact category $\sB$ has \emph{enough
homotopy projective complexes of projectives} if for any complex $T^\bu$
in $\sB$ there exists a homotopy projective complex \emph{of projective
objects} $P^\bu$ together with a quasi-isomorphism $P^\bu\rarrow T^\bu$.
 If this is the case, then the above-mentioned left adjoint functor
$\sD(\sB)\rarrow\Hot(\sB)$ to the canonical Verdier quotient functor
factorizes through the homotopy category of complexes of projective
objects, $\sD(\sB)\rarrow\Hot(\sB_\proj)\rarrow\Hot(\sB)$.
 We will denote the resulting functor by $\varkappa\:\sD(\sB)\rarrow
\Hot(\sB_\proj)$.

 For example, any locally presentable abelian category $\sB$ with
enough projective objects has enough homotopy projective complexes of
projectives~\cite[Lemma~6.1 and Corollary~6.7]{PS4}.
 In particular, for any coassociative coalgebra $\D$ over a field~$k$,
the abelian category of left $\D$\+contramodules $\sB=\D\contra$ has
enough homotopy projective complexes of projective
objects~\cite[Sections~2.4 and~5.5]{Pkoszul}
(see also~\cite[Theorem~1.1(a)]{Pmc}).

\begin{thm} \label{weak-pseudoderived}
\textup{(a)} Let\/ $\sA$ be an exact category with enough homotopy
injective complexes of injectives, and let\/ $\sE\subset\sA$ be
a coresolving subcategory containing the injective objects.
 Then the triangulated functor\/ $\sD(\sE)\rarrow\sD(\sA)$ induced
by the inclusion of exact categories\/ $\sE\rarrow\sA$ is a Verdier
quotient functor, and it has a right adjoint functor
$\rho\:\sD(\sA)\rarrow\sD(\sE)$.
 Moreover, there is a diagram of triangulated functors
$$
\begin{tikzcd}
\Hot(\sA_\inj) \arrow[d] \arrow[dd, two heads, bend right = 50] \\
\sD(\sE) \arrow[d, two heads] \\
\sD(\sA) \arrow[u, tail, bend right = 51, "\rho"]
\arrow[uu, tail, bend right = 66, "\theta"]
\end{tikzcd}
$$
where the functor\/ $\Hot(\sA_\inj)\rarrow\sD(\sE)$ is induced by
the inclusion\/ $\sA_\inj\rarrow\sE$, the functors shown by two-headed
arrows are Verdier quotient functors, the long downwards directed
curvilinear arrow is the composition of two straight downwards directed
arrows, and the upwards directed curvilinear arrows with a tail show
fully faithful right adjoint functors to the downwards directed arrows.
 The functor $\rho\:\sD(\sA)\rarrow\sD(\sE)$ is the composition\/
$\sD(\sA)\overset\theta\rarrow \sD(\sA_\inj)\rarrow\sD(\sE)$. \par
\textup{(b)} Let\/ $\sB$ be an exact category with enough homotopy
projective complexes of projectives, and let\/ $\sF\subset\sB$ be
a resolving subcategory containing the projective objects.
 Then the triangulated functor\/ $\sD(\sF)\rarrow\sD(\sB)$ induced
by the inclusion of exact categories\/ $\sF\rarrow\sB$ is a Verdier
quotient functor, and it has a left adjoint functor
$\lambda\:\sD(\sB)\rarrow\sD(\sF)$.
 Moreover, there is a diagram of triangulated functors
$$
\begin{tikzcd}
\Hot(\sB_\proj) \arrow[d] \arrow[dd, two heads, bend left = 50] \\
\sD(\sF) \arrow[d, two heads] \\
\sD(\sB) \arrow[u, tail, bend left = 51, "\lambda"']
\arrow[uu, tail, bend left = 66, "\varkappa"']
\end{tikzcd}
$$
where the functor\/ $\Hot(\sB_\proj)\rarrow\sD(\sF)$ is induced by
the inclusion\/ $\sB_\proj\rarrow\sF$, the functors shown by two-headed
arrows are Verdier quotient functors, the long downwards directed
curvilinear arrow is the composition of two straight downwards directed
arrows, and the upwards directed curvilinear arrows with a tail show
fully faithful left adjoint functors to the downwards directed arrows.
 The functor $\lambda\:\sD(\sB)\rarrow\sD(\sF)$ is the composition\/
$\sD(\sB)\overset\varkappa\rarrow\sD(\sB_\proj)\rarrow\sD(\sF)$.
\end{thm}

\begin{proof}
 We will prove part~(a), as part~(b) is dual.
 We have already explained how the functors are constructed.
 Any triangulated functor with a fully faithful adjoint is a Verdier
quotient functor, and conversely any triangulated functor adjoint
to a Verdier quotient functor is
fully faithful~\cite[Proposition~I.1.3]{GZ}.

 The adjunction between the two functors $\Hot(\sA_\inj)\rarrow
\sD(\sA)$ and $\sD(\sA)\overset\theta\rarrow\Hot(\sA_\inj)$ is obtained
by restricting the adjunction between the two functors $\Hot(\sA)
\rarrow\sD(\sA)$ and $\sD(\sA)\rarrow\Hot(\sA)$ to the full
subcategory $\Hot(\sA_\inj)\subset\Hot(\sA)$.
 Since the image of the functor $\sD(\sA)\rarrow\Hot(\sA)$ is
contained in $\Hot(\sA_\inj)\subset\Hot(\sA)$, such restriction
makes sense.
 The functor $\Hot(\sA)\rarrow\sD(\sA)$ is a Verdier quotient functor,
hence the functor $\sD(\sA)\rarrow\Hot(\sA)$ is fully faithful, so
the functor $\sD(\sA)\overset\theta\rarrow\Hot(\sA_\inj)$ is fully
faithful as well, and it follows that the functor $\Hot(\sA_\inj)
\rarrow\sD(\sA)$ is a Verdier quotient functor. {\hbadness=1100\par}

 It remains to show that the functor $\rho\:\sD(\sA)\rarrow\sD(\sE)$
is fully faithful and right adjoint to the triangulated functor
$\sD(\sE)\rarrow \sD(\sA)$ induced by the inclusion $\sE\rarrow\sA$;
then it will follow that the latter functor is a Verdier quotient
functor.
 For this purpose, we decompose the functor $\Hot(\sA_\inj)\rarrow
\sD(\sE)$ as $\Hot(\sA_\inj)\rarrow\Hot(\sE)\rarrow\sD(\sE)$, where
$\Hot(\sA_\inj)\rarrow\Hot(\sE)$ is the functor induced by
the inclusion of additive categories $\sA_\inj\rarrow\sE$ and
$\Hot(\sE)\rarrow\sD(\sE)$ is the canonical Verdier quotient functor.
 Then the functor~$\rho$ decomposes as
$$
 \sD(\sA)\overset\theta\lrarrow\Hot(\sA_\inj)\lrarrow
 \Hot(\sE)\lrarrow\sD(\sE).
$$

 For simplicity, in the following lemma we denote the compositions
$\sD(\sA)\overset\theta\rarrow\Hot(\sA_\inj)\rarrow\Hot(\sE)$ and
$\sD(\sA)\overset\theta\rarrow\Hot(\sA_\inj)\rarrow\Hot(\sA)$
also by~$\theta$.

\begin{lem} \label{triangulated-lemma}
 Let\/ $\sH$ be a triangulated category and\/ $\sY\subset\sG\subset
\sH$, \,$\sY\subset\sX\subset\sH$ be its full triangulated subcategories.
 Suppose that the Verdier quotient functor\/ $\sH\rarrow\sH/\sX$ has
a right adjoint $\theta\:\sH/\sX\rarrow\sH$, whose image is contained
in\/~$\sG$.
 Then the triangulated functor\/ $\sG/\sY\rarrow\sH/\sX$ induced by
the inclusion $(\sG,\sY)\hookrightarrow(\sH,\sX)$ is a Verdier quotient
functor with a right adjoint functor~$\rho$, which can be computed as
the composition\/ $\sH/\sX\overset\theta\rarrow\sG\rarrow\sG/\sY$.
\end{lem}

\begin{proof}
 The functor~$\theta$ is fully faithful as an adjoint to a Verdier
quotient functor.
 Denote by~$\rho$ the composition $\sH/\sX\rarrow\sG\rarrow\sG/\sY$;
so $\rho(C)=\theta(C)/\sY\in\sG/\sY$ for all $C\in\sH/\sX$.
 To prove that the functor~$\rho$ is fully faithful, it
suffices to check that all the objects in the image of~$\theta$ are
right orthogonal to $\sY$, that is $\Hom_\sG(Y,\theta(C))=0$
for all $C\in\sH/\sX$ and $Y\in\sY$.
 Indeed, we have $\Hom_\sH(X,\theta(C))=0$ for all $X\in\sX$, as
the image of $X$ in $\sH/\sX$ vanishes.

 It remains to show that~$\rho$ is right adjoint to the functor
$\sG/\sY\rarrow\sH/\sX$; it will then follow that the latter is
a Verdier quotient functor.
 Let $E\in\sG$ and $C\in\sH/\sX$ be two objects.
 Then we have
\begin{multline*}
 \Hom_{\sG/\sY}(E/\sY,\rho(C))=\Hom_{\sG/\sY}
 (E/\sY,\theta(C)/\sY)\simeq\Hom_\sG(E,\theta(C)) \\ = 
 \Hom_\sH(E,\theta(C))\simeq\Hom_{\sH/\sX}(E/\sX,C),
\end{multline*}
since $\theta(C)$ is right orthogonal to~$\sY$.
\end{proof}

 To finish the proof of part~(a), it remains to set
$$
 \sH=\Hot(\sA)\,\supset\,\sG=\Hot(\sE)
$$
and denote further by $\sX\subset\sH$ the full subcategory of acyclic
complexes in the exact category $\sA$ and by $\sY\subset\sX\cap\sG$
the full subcategory of acyclic complexes in the exact category~$\sE$.
 So $\sH/\sX=\sD(\sA)$ and $\sG/\sY=\sD(\sE)$.
 The right adjoint functor to the Verdier quotient functor
$\Hot(\sA)\rarrow\sD(\sA)$ lands inside $\Hot(\sA_\inj)\subset
\Hot(\sE)\subset\Hot(\sA)$, as required in the lemma.
\end{proof}

\Section{Quasi-Finiteness Conditions for Coalgebras}
\label{quasi-finite-secn}

 We refer to the classical book~\cite{Swe}, the introductory section
and appendix~\cite[Section~0.2 and Appendix~A]{Psemi},
the memoir~\cite{Pkoszul}, the overview~\cite{Prev},
the paper~\cite{Pmc}, and the references therein for a general
discussion of coassociative coalgebras over fields and module objects
(comodules and contramodules) over them.

 Let $k$~be a fixed ground field and $\C$ be a coassociative
coalgebra (with counit) over~$k$.
 We denote by $\C\comodl$ and $\comodr\C$ the abelian categories of
left and right $\C$\+comodules.
 The abelian category of left $\C$\+contramodules is denoted by
$\C\contra$.

 For any two left $\C$\+comodules $\M$ and $\N$, we denote by
$\Hom_\C(\M,\N)$ the $k$\+vector space of left $\C$\+comodule morphisms
$\M\rarrow\N$.
 For any two left $\C$\+contramodules $\fS$ and $\fT$, we denote by
$\Hom^\C(\fS,\fT)$ the $k$\+vector space of left $\C$\+contramodule
morphisms $\fS\rarrow\fT$.
 The coalgebra opposite to $\C$ is denoted by $\C^\rop$; so a right
$\C$\+comodule is the same thing as a left $\C^\rop$\+comodule.

 We recall that for any right $\C$\+comodule $\N$ and $k$\+vector
space $V$ the vector space $\Hom_k(\N,V)$ has a natural left
$\C$\+contramodule structure~\cite[Sections~1.1\+-2]{Prev}.
 We refer to~\cite[Section~2]{Pmc}, \cite[Section~3.1]{Prev},
\cite[Section~2.2]{Pkoszul}, or~\cite[Sections~0.2.6
and~5.1.1\+-2]{Psemi} for the definition and discussion of
the functor of \emph{contratensor product} $\N\ocn_\C\fT$ of
a right $\C$\+comodule $\N$ and a left $\C$\+contramodule~$\fT$.

 The construction of the \emph{cotensor product} $\N\oc_\C\M$ of
a right $\C$\+comodule $\N$ and a left $\C$\+comodule $\M$ goes
back at least to the paper~\cite[Section~2]{MM}.
 The dual-analogous construction involving contramodules is
the vector space of \emph{cohomomorphisms} $\Cohom_\C(\M,\fT)$
from a left $\C$\+comodule $\M$ to a left $\C$\+contramodule~$\fT$.
 We refer to~\cite[Section~2]{Pmc}, \cite[Sections~2.5\+-6]{Prev},
or~\cite[Sections~0.2.1, 0.2.4, 1.2.1, and~3.2.1]{Psemi} for
the definitions and discussion of these constructions.

\medskip

 Finiteness and quasi-finiteness conditions for coalgebras and
comodules were studied in~\cite{Tak,WW,GTNT,Pmc} and many other papers.
 The next two lemmas are well-known and included here for the reader's
convenience.

 Given a subcoalgebra $\B\subset\C$ and a left $\C$\+comodule $\M$,
let ${}_\B\M\subset\M$ denote the maximal $\C$\+subcomodule in $\M$
whose $\C$\+comodule structure comes from a $\B$\+comodule structure.
 The $\B$\+comodule ${}_\B\M$ can be computed as the full preimage
of the subcomodule $\B\ot_k\M\subset\C\ot_k\M$ under the coaction map
$\M\rarrow\C\ot_k\M$, or as the cotensor product $\B\oc_\C\M$\,
\cite[Section~2]{Pmc}.

\begin{lem} \label{qf-cogenerated-comodule-definition-lem}
 Let\/ $\C$ be a coassociative coalgebra over~$k$ and\/ $\M$ be
a left\/ $\C$\+comodule.
 Then the following four conditions are equivalent:
\begin{itemize}
\item for any finite-dimensional subcoalgebra\/ $\B\subset\C$,
the $k$\+vector space\/ ${}_\B\M$ is finite-dimensional;
\item for any cosimple subcoalgebra\/ $\A\subset\C$,
the $k$\+vector space\/ ${}_\A\M$ is finite-dimensional;
\item for any finite-dimensional left\/ $\C$\+comodule\/ $\K$,
the $k$\+vector space\/ $\Hom_\C(\K,\M)$ is finite-dimensional.
\item for any irreducible left\/ $\C$\+comodule\/ $\I$,
the $k$\+vector space\/ $\Hom_\C(\I,\M)$ is finite-dimensional.
\end{itemize}
\end{lem}

\begin{proof}
 This is essentially a statement about comodules over finite-dimensional
coalgebras, or which is the same thing, modules over finite-dimensional
algebras.
 Basically, the assertion is that a module over a finite-dimensional
algebra is finite-dimensional if and only if its socle is
finite-dimensional.

 We refer to~\cite[Section~2]{Swe} for the background material.
 In particular, one should keep in mind that $\C$ is the union of its
finite-dimensional subcoalgebras, all $\C$\+comodules are the unions of
their finite-dimensional subcomodules, and all finite-dimensional
$\C$\+comodules are comodules over finite-dimensional subcoagebras of
$\C$; so irreducible left $\C$\+comodules correspond bijectively to
cosimple subcoalgebras in~$\C$.
\end{proof}

 We will say that a left $\C$\+comodule $\M$ is \emph{quasi-finitely
cogenerated} if it satisfies the equivalent conditions of
Lemma~\ref{qf-cogenerated-comodule-definition-lem}.
 (Such comodules were called ``quasi-finite'' in~\cite{Tak,GTNT}.)
 Recall that a left $\C$\+comodule is said to be
\emph{finitely cogenerated}~\cite{Tak,WW,Pmc} if it is a subcomodule of
a cofree left $\C$\+comodule $\C\ot_k V$ with a finite-dimensional
vector space of cogenerators~$V$.
 Any finitely cogenerated $\C$\+comodule is quasi-finitely cogenerated,
while the cofree left $\C$\+comodule $\C\ot_kV$ cogenerated by
an infinite-dimensional $k$\+vector space $V$ is not quasi-finitely
cogenerated when $\C\ne0$.

 One can see from~\cite[Lemma~2.2(e)]{Pmc} that the classes of
finitely cogenerated and quasi-finitely cogenerated left
$\C$\+comodules coincide if and only if the maximal cosemisimple
subcoalgebra $\C^\ss$ of the coalgebra $\C$ is finite-dimensional
(cf.~\cite[Proposition~1.6 in journal version or Proposition~2.5 in
\texttt{arXiv} version]{GTNT}).
 Unlike the finite cogeneratedness condition, the quasi-finite
cogeneratedness condition on comodules is \emph{Morita invariant},
i.~e., it is preserved by equivalences of the categories of comodules
$\C\comodl\simeq\D\comodl$ over different coalgebras $\C$ and
$\D$\,~\cite{Tak}.

\begin{lem} \label{qf-cogenerated-comodules-properties}
\textup{(a)} The class of all quasi-finitely cogenerated left\/
$\C$\+comodules is closed under extensions and the passages to
arbitrary subcomodules. \par
\textup{(b)} Any quasi-finitely cogenerated\/ $\C$\+comodule is
a subcomodule of a quasi-finitely cogenerated injective\/
$\C$\+comodule. \qed
\end{lem}

\begin{proof}
 To prove part~(a), notice that the functor $\M\longmapsto{}_\B\M$ is
left exact for any subcoalgebra $\B\subset\C$.
 In part~(b), it suffices to say that the injective envelope of
a quasi-finitely cogenerated comodule is quasi-finitely cogenerated.
 Indeed, if $\M$ is a left $\C$\+comodule, $\J$ is an injective
envelope of $\M$, and $\A$ is a co(semi)simple subcoalgebra in $\C$,
then any $\A$\+subcomodule in $\J$ is contained in~$\M$.
\end{proof}

 The following result can be found in~\cite[Theorem~2.1 in journal
version or Theorem~3.1 in \texttt{arXiv} version]{GTNT}.

\begin{prop} \label{quasi-co-Noetherian-prop}
 Let\/ $\C$ be a coassociative coalgebra over a field~$k$.
 Then the following four conditions are equivalent:
\begin{itemize}
\item any quotient comodule of a quasi-finitely cogenerated
left\/ $\C$\+comodule is quasi-finitely cogenerated;
\item any quotient comodule of a quasi-finitely cogenerated injective
left\/ $\C$\+comodule is quasi-finitely cogenerated;
\item any quotient comodule of a finitely cogenerated left\/
$\C$\+comodule is quasi-finitely cogenerated;
\item any quotient comodule of the left\/ $\C$\+comodule $\C$ is
quasi-finitely cogenerated.
\end{itemize}
\end{prop}

\begin{proof}
 We will prove that the fourth condition implies the second one
(the other implications being obvious in view of
Lemma~\ref{qf-cogenerated-comodules-properties}(b)).

 As in any locally Noetherian Grothendieck abelian category, every
injective $\C$\+comodule is a direct sum of indecomposable injectives.
 In fact, the category $\C\comodl$ is even locally finite (any object
is the union of its subobjects of finite length); hence an injective
$\C$\+comodule is indecomposable if and only if its socle is irreducible.
 The correspondence assigning to an injective left $\C$\+comodule $\J$
its socle $\K$ restricts to a bijection between the isomorphism classes
of indecomposable injective left $\C$\+comodules $\J_i$ and
the isomorphism classes of irreducible left $\C$\+comodules~$\I_i$.
 The latter corerespond bijectively to the cosimple subcoalgebras
$\A_i\subset\C$.
 It follows that an injective left $\C$\+comodule $\J$ is quasi-finitely
cogenerated if and only if its socle $\K$ contains any irreducible
left $\C$\+comodule $\I_i$ with at most finite multiplicity.

 Let $\J$ be a quasi-finitely cogenerated injective left $\C$\+comodule.
 Consider a decomposition of $\J$ ito a direct sum of indecomposable
injective left $\C$\+comodules $\J_a$, choose a well-ordering of
the set of indices~$a$, and consider the ordinal-indexed increasing
filtration of $\J$ associated with this direct sum decomposition and
this ordering of the summands.
 Then any subcomodule $\L$ and any quotient comodule $\M$ of
the $\C$\+comodule $\J$ acquires the induced increasing filtration
indexed by the same ordinal.
 The successive quotient comodules $\M_a$ of the induced filtration
on $\M$ are certain quotients of the indecomposable injectives~$\J_a$.

 This argument shows that it suffices to check quasi-finite
cogeneratedness of the quotient comodules $\M=\J/\L$ of the form
$\M=\bigoplus_a\M_a$, where $\M_a$ are certain quotient comodules
of~$\J_a$.
 We arrive to the following criterion.
 All quotient comodules of quasi-finitely cogenerated injective left
$\C$\+comodules are quasi-finitely cogenerated if and only if both
of the next two conditions hold:
\begin{itemize}
\item[(${*}$)] for any cosimple subcoalgebra $\B\subset\C$ and
any quotient comodule $\M$ of an indecomposable injective
left  $\C$\+comodule $\J_i$, the subcomodule ${}_\B\M\subset\M$ is
finite-dimensional;
\item[(${*}{*}$)] for any cosimple subcoalgebra $\B\subset\C$, there
exist at most finite number of isomorphism classes of indecomposable
injective left $\C$\+comodules $\J_i$ for which $\J_i$ has a quotient
comodule $\M$ with ${}_\B\M\ne0$.
\end{itemize}

 Now let us consider the direct sum $\J=\bigoplus_i\J_i$ of all
the indecomposable injective left $\C$\+comodules, exactly one copy
of each.
 Then $\J$ is a direct summand of the left $\C$\+comodule~$\C$.
 If at least one of the conditions~(${*}$) and~(${*}{*}$) is \emph{not}
satisfied, then $\J$ has a quotient comodule $\M=\bigoplus_i\M_i$
which is not quasi-finitely cogenerated.
 This observation finishes the proof.
\end{proof}

 A coalgebra $\C$ is said to be \emph{left quasi-co-Noetherian} if it
satisfies the equivalent conditions of
Proposition~\ref{quasi-co-Noetherian-prop}, i.~e., if any quotient
comodule of a quasi-finitely cogenerated left $\C$\+comodule is
quasi-finitely cogenerated.
 (Such coalgebras were called ``left strictly quasi-finite''
in~\cite{GTNT}.)
 Over a left quasi-co-Noetherian coalgebra $\C$, quasi-finitely
cogenerated left $\C$\+comodules form an abelian category.
 Recall that a coalgebra $\C$ is said to be \emph{left
co-Noetherian}~\cite{WW,GTNT,Pmc} if any quotient comodule of a finitely
cogenerated left $\C$\+comodule is finitely cogenerated.
 It is clear from Proposition~\ref{quasi-co-Noetherian-prop} that
any left co-Noetherian coalgebra is left quasi-co-Noetherian.

 An example of a quasi-co-Noetherian coalgebra that is not
co-Noetherian is given in~\cite[Example~1.5 in journal version or
Example~2.3 in \texttt{arXiv} version]{GTNT}.
 An example of a right co-Noetherian coalgebra that is not
left quasi-co-Noetherian can be found in~\cite[Example~2.7 in journal
version or Example~3.6 in \texttt{arXiv} version]{GTNT}. 

 A $\C$\+comodule $\M$ is said to be \emph{quasi-finitely
copresented} if it is isomorphic to the kernel of a morphism of
quasi-finitely cogenerated injective $\C$\+comodules.
 Any finitely copresented $\C$\+comodule in the sense of~\cite{WW}
and~\cite[Section~2]{Pmc} is quasi-finitely copresented.
 Any quasi-finitely copresented $\C$\+comodule is quasi-finitely
cogenerated.

 An example of a quasi-finitely cogenerated comodule that is not
quasi-finitely copresented can be easily extracted
from~\cite[Example~2.1 in \texttt{arXiv} version]{GTNT}
using part~(c) of the next lemma.

\begin{lem}
\textup{(a)} The kernel of a morphism from a quasi-finitely
copresented\/ $\C$\+comodule to a quasi-finitely cogenerated one is
quasi-finitely copresented. \par
\textup{(b)} The class of quasi-finitely copresented\/ $\C$\+comodules
is closed under extensions. \par
\textup{(c)} The cokernel of an injective morphism from
a quasi-finitely copresented\/ $\C$\+comodule to a quasi-finitely
cogenerated one is quasi-finitely cogenerated.
\end{lem}

\begin{proof}
 Follows from Lemma~\ref{qf-cogenerated-comodules-properties}
(cf.\ the proof of~\cite[Lemma~2.8(a)]{Pmc}).
\end{proof}

 Given a subcoalgebra $\B\subset\C$ and a left $\C$\+contramodule $\fT$,
we denote by ${}^\B\fT$ the maximal quotient contramodule of $\fT$
whose $\C$\+contramodule structure comes from a $\B$\+contramodule
structure.
 The $\B$\+contramodule ${}^\B\fT$ can be computed as the cokernel of
the composition $\Hom_k(\C/\B,\fT)\rarrow\fT$ of the natural embedding
$\Hom_k(\C/\B,\fT)\rarrow\Hom_k(\C,\fT)$ with the contraaction map
$\Hom_k(\C,\fT)\rarrow\fT$, or as the space of cohomomorphisms
${}^\B\fT=\Cohom_\C(\B,\fT)$\, \cite[Section~2]{Pmc}.

\begin{lem} \label{qf-generated-contramodule-definition-lem}
 Let\/ $\C$ be a coassociative coalgebra over~$k$ and\/ $\fT$ be
a left\/ $\C$\+contra\-module.
 Then the following four conditions are equivalent:
\begin{itemize}
\item for any finite-dimensional subcoalgebra\/ $\B\subset\C$,
the $k$\+vector space\/ ${}^\B\fT$ is finite-dimensional;
\item for any cosimple subcoalgebra\/ $\A\subset\C$,
the $k$\+vector space\/ ${}^\A\fT$ is finite-dimensional;
\item for any finite-dimensional left\/ $\C$\+contramodule\/ $\fK$,
the $k$\+vector space\/ $\Hom^\C(\fT,\fK)$ is finite-dimensional;
\hbadness=2825
\item for any irreducible left\/ $\C$\+contramodule\/ $\fI$,
the $k$\+vector space\/ $\Hom^\C(\fT,\fI)$ is finite-dimensional.
\end{itemize}
\end{lem}

\begin{proof}
 This is also essentially a statement about contramodules over
finite-dimensional coalgebras, or which is the same thing, modules
over finite-dimensional algebras.
 Basically, the assertion is that a module over a finite-dimensional
algebra is finite-dimensional if and only if its quotient module by
its cosocle is finite-dimensional. {\hbadness=1800\par}

 We refer to~\cite[Appendix~A]{Psemi} and~\cite[Section~1]{Prev}
for the background material.
 In particular, one has to be careful in that \emph{not} every
$\C$\+contramodule embeds into the projective limit of its
finite-dimensional quotient contramodules; nevertheless, any
nonzero $\C$\+contramodule has a nonzero finite-dimensional
quotient contramodule, and therefore all irreducible $\C$\+contramodules
are finite-dimensional.
 Furthermore, \emph{not} every finite-dimensional $\C$\+contramodule
is a contramodule over a finite-dimensional subcoalgebra of $\C$,
generally speaking; but every irreducible $\C$\+contramodule is,
so irreducible left $\C$\+contramodules still correspond bijectively
to cosimple subcoalgebras in~$\C$.
\end{proof}

 We will say that a left $\C$\+contramodule $\fT$ is
\emph{quasi-finitely generated} if it satisfies the equivalent
conditions of Lemma~\ref{qf-generated-contramodule-definition-lem}.
 Recall that a left $\C$\+contramodule is said to be \emph{finitely
generated}~\cite[Section~2]{Pmc} if it is a quotient contramodule of
a free left $\C$\+contramodule $\Hom_k(\C,V)$ with a finite-dimensional
space of generators~$V$.
 According to~\cite[Lemma~2.5(a) and the proof of Lemma~2.5(b)]{Pmc},
any finitely generated $\C$\+contramodule is quasi-finitely generated,
while the free left $\C$\+contramodule $\Hom_k(\C,V)$ generated by
an infinite-dimensional vector space $V$ is not quasi-finitely
generated when $\C\ne0$.

 One can see from~\cite[Lemma~2.5(e)]{Pmc} that the classes of
finitely generated and quasi-finitely generated left $\C$\+contramodules
coincide if and only if the maximal cosemisimple subcoalgebra $\C^\ss$
of the coalgebra $\C$ is finite-dimensional.
 Unlike the finite generatedness condition, the quasi-finite
generatedness condition on contramodules is Morita invariant, i.~e.,
it is preserved by equivalences of the categories of contramodules
$\C\contra\simeq\D\contra$ over different coalgebras $\C$ and $\D$
(see~\cite[Section~7.5.3]{Psemi} for a discussion).

\begin{lem} \label{qf-generated-contramodules-properties}
\textup{(a)} The class of quasi-finitely generated left\/
$\C$\+contramodules is closed under extensions and the passages to
arbitrary quotient contramodules. \par
\textup{(b)} Any quasi-finitely generated left\/ $\C$\+contramodule
is a quotient contramodule of a quasi-finitely generated projective\/
$\C$\+contramodule.
\end{lem}

\begin{proof}
 To prove part~(a), notice that the functor $\fT\longmapsto{}^\B\fT$ is
right exact for any subcoalgebra $\B\subset\C$.

 The proof of part~(b) is based on the arguments in the first half of
the proof of \cite[Lemma~A.3]{Psemi}.
 Given a left $\C$\+contramodule $\fT$, one considers the left
$\C^\ss$\+contramodule $\fK={}^{\C^\ss}\fT$.
 The key step is to construct for any left $\C^\ss$\+contramodule $\fK$
a projective left $\C$\+contramodule $\P$ such that the left
$\C^\ss$\+contramodule ${}^{\C^\ss}\P$ is isomorphic to~$\fK$.
 Then one applies the contramodule Nakayama lemma~\cite[Lemma~A.2.1]{Psemi}
in order to show that $\fT$ is a quotient $\C$\+contramodule of~$\P$.

 The above proof of part~(b) looks different from the proof of
Lemma~\ref{qf-cogenerated-comodules-properties}(b), but in fact they
are very similar (or dual-analogous).
 Any $\C$\+contramodule has a projective
cover~\cite[Example~11.3]{Pproperf}, and the $\C$\+contramodule $\P$ in
the above argument is a projective cover of the $\C$\+contramodule $\fT$
(cf.\ the construction in~\cite[Section~9]{Pproperf}).
\end{proof}

 A $\C$\+contramodule $\fT$ is said to be \emph{quasi-finitely
presented} if it is isomorphic to the cokernel of a morphism of
quasi-finitely generated projective $\C$\+contramodules.
 Any finitely presented contramodule in the sense
of~\cite[Section~2]{Pmc} is quasi-finitely presented.
 Any quasi-finitely presented $\C$\+contramodule is quasi-finitely
generated.

\begin{lem}
\textup{(a)} The cokernel of a morphism from a quasi-finitely
generated\/ $\C$\+contramodule to a quasi-finitely presented one
is quasi-finitely presented. \par
\textup{(b)} The class of quasi-finitely presented\/ $\C$\+contramodules
is closed under extensions. \par
\textup{(c)} The kernel of a surjective morphism from a quasi-finitely
generated\/ $\C$\+contra\-module to a quasi-finitely presented one is
quasi-finitely generated.
\end{lem}

\begin{proof}
 Follows from Lemma~\ref{qf-generated-contramodules-properties}
(cf.~\cite[Lemma~1.1]{Pfp} and~\cite[Lemma~2.8(b)]{Pmc}).
\end{proof}

 The following proposition is our version
of~\cite[Proposition~2.9]{Pmc}.

\begin{prop} \label{coalgebra-dualization-isomorphism-prop}
\textup{(a)} The functor\/ $\N\longmapsto\N^*=\Hom_k(\N,k)$ restricts
to an anti-equivalence between the additive category of
quasi-finitely copresented right\/ $\C$\+comodules and the additive
category of quasi-finitely presented left\/ $\C$\+contramodules. \par
\textup{(b)} For any right\/ $\C$\+comodule\/ $\M$, any quasi-finitely
cogenerated right\/ $\C$\+comodule\/ $\N$, and any $k$\+vector space
$V$, the natural $k$\+linear map\/ $\Hom_k(\M,\>\N\ot_kV)\rarrow
\Hom_k(\N^*,\Hom_k(\M,V))$ restricts to an isomorphism of the Hom
spaces in the categories of right\/ $\C$\+comodules and left\/
$\C$\+contramodules
$$
 \Hom_{\C^\rop}(\M,\>\N\ot_kV)\.\simeq\.\Hom^\C(\N^*,\Hom_k(\M,V)).
$$ \par
\textup{(c)} For any right\/ $\C$\+comodule\/ $\M$, any quasi-finitely
cogenerated right\/ $\C$\+comodule\/ $\N$, and any $k$\+vector space
$V$, the natural $k$\+linear map\/ $(\M\ot_k\N^*)\ot_kV\rarrow
\M\ot_k\Hom_k(\N,V)$ induces an isomorphism of the (contra)tensor
product spaces
$$
 (\M\ocn_\C\N^*)\ot_kV\.\simeq\.\M\ocn_\C\Hom_k(\N,V).
$$
\end{prop}

\begin{proof}
 Part~(b): for a right $\C$\+comodule $\M$ and a subcoalgebra
$\B\subset\C$, we denote the maximal subcomodule of $\M$ whose
$\C$\+comodule structure comes from a $\B$\+comodule structure
by~$\M_\B$.
 Then for any $k$\+vector space $V$ we have $\.{}^\B\!\Hom_k(\M,V)
\allowbreak=\Hom_k(\M_\B,V)$.
 Since any right $\C$\+comodule $\M$ is the union of its subcomodules
$\M_\A$ over the finite-dimensional subcoalgebras $\A\subset\C$, it
follows that
$$
 \Hom_k(\M,V)=\varprojlim\nolimits_\A {}^\A\!\Hom_k(\M,V).
$$
 Therefore,
\begin{multline*}
 \Hom^\C(\N^*,\Hom_k(\M,V))=
 \varprojlim\nolimits_\A\.\Hom^\C(\N^*,{}^\A\!\Hom_k(\M,V)) \\
 =\varprojlim\nolimits_\A\.\Hom^\C({}^\A(\N^*),{}^\A\!\Hom_k(\M,V))
 =\varprojlim\nolimits_\A\.\Hom^\A((\N_\A)^*,\Hom_k(\M_\A,V))\\
 \simeq\varprojlim\nolimits_\A\.\Hom_{\A^\rop}(\M_\A,\>\N_\A\ot_kV)=
 \Hom_{\C^\rop}(\M,\>\N\ot_kV)
\end{multline*}
because the right $\A$\+comodule $\N_\A$ is finite-dimensional.

 To prove part~(a), we notice from the computations above that the left
$\C$\+contra\-module $\N^*$ is quasi-finitely generated if and only if
a right $\C$\+comodule $\N$ is quasi-finitely cogenerated.
 Furthermore, substituting $V=k$ into the assertion~(b), we see that
the dualization functor $\N\longmapsto\N^*\:\comodr\C\rarrow\C\contra$
is fully faithful on the full subcategory of quasi-finitely cogenerated
comodules in $\comodr\C$.
 It remains to prove the essential surjectivity.

 As the left $\C$\+contramodule $\J^*$ is projective for any injective
right $\C$\+comodule $\J$, the dualization functor takes quasi-finitely
cogenerated injective right $\C$\+comodules to quasi-finitely generated
projective left $\C$\+contramodules.
 A projective left $\C$\+contra\-module $\P$ is uniquely determined,
up to isomorphism, by the left $\C^\ss$\+contramodule ${}^{\C^\ss}\P$,
and all the quasi-finitely generated $\C^\ss$\+contramodules belong
to the image of the dualization functor; therefore so do all
the quasi-finitely generated projective $\C$\+contramodules
(cf.\ the proofs of
Lemmas~\ref{qf-cogenerated-comodules-properties}(b)
and~\ref{qf-generated-contramodules-properties}(b)).

 Finally, any quasi-finitely presented $\C$\+contramodule is
the cokernel of a morphism of quasi-finitely generated projective
$\C$\+contramodules, this morphism comes from a morphism of
quasi-finitely cogenerated injective $\C$\+comodules, the kernel
of the latter morphism is a quasi-finitely copresented $\C$\+comodule,
and the dualization functor takes the kernels to the cokernels.

 Part~(c): For any right $\C$\+comodule $\M$ and left $\C$\+contramodule
$\fT$, one has $\M\ocn_\C\fT=(\varinjlim_\A\M_A)\ocn_\C\fT=
\varinjlim_\A(\M_\A\ocn_\C\fT)=\varinjlim_\A(\M_\A\ocn_\A{}^\A\fT)$,
where the inductive limit is taken over all the finite-dimensional
subcoalgebras $\A\subset\C$.
 In particular,
\begin{multline*}
 \M\ocn_\C\Hom_k(\N,V)=
 \varinjlim\nolimits_\A(\M_\A\ocn_\C\Hom_k(\N_\A,V)) \\
 \simeq\varinjlim\nolimits_\A((\M_\A\ocn_\C\N_\A^*)\ot_kV)=
 (\M\ocn_\C\N^*)\ot_kV,
\end{multline*}
because $\N_\A$ is finite-dimensional.
\end{proof}

\begin{rem} \label{takeuchi-co-hom-explained}
 The following construction of a (partially defined) functor of two
comodule arguments plays a major role in the classical paper~\cite{Tak}.
 Let $\N$ be a quasi-finitely cogenerated right $\C$\+comodule.
 Then, for any right $\C$\+comodule $\M$, the vector space of comodule
homomorphisms $\Hom_{\C^\rop}(\M,\N)$ is the projective limit of
finite-dimensional vector spaces.
 Hence it is naturally the dual vector space to a certain vector space,
which is called ``co-hom'' and denoted by $h_{\textrm-\C}(\N,\M)$
in~\cite{Tak}.
 The vector space $h_{\textrm-\C}(\N,\M)$ is characterized by
the property that, for any $k$\+vector space $V$, there is a natural
adjunction isomorphism of $k$\+vector spaces $\Hom_{\C^\rop}(\M,\>
\N\ot_kV)\simeq\Hom_k(h_{\textrm-\C}(\N,\M),V)$.

 Proposition~\ref{coalgebra-dualization-isomorphism-prop} explains
that the functor $h_{\textrm-\C}$ can be expressed in our terms as
$h_{\textrm-\C}(\N,\M)=\M\ocn_\C\N^*$.
 Indeed, for any $k$\+vector space $V$ we have
$$
 \Hom_k(\M\ocn_\C\N^*,\>V)\simeq\Hom^\C(\N^*,\Hom_k(\M,V))
 \simeq\Hom_{\C^\rop}(\M,\>\N\ot_kV)
$$
by part~(b) of the proposition.
\end{rem}

 A coalgebra $\C$ is called \emph{right quasi-cocoherent} if any
quasi-finitely cogenerated quotient comodule of a quasi-finitely
copresented right $\C$\+comodule is quasi-finitely copresented, or
equivalently, if any quasi-finitely generated subcontramodule of
a quasi-finitely presented left $\C$\+contramodule is quasi-finitely
presented.
 Over a right co-coherent coalgebra $\C$, the categories of
quasi-finitely copresented right comodules and quasi-finitely
presented left contramodules are abelian.
 Any left quasi-co-Noetherian coalgebra is left quasi-cocoherent,
and any quasi-finitely cogenerated left comodule over such a coalgebra
is quasi-finitely copresented.

 Recall that a coalgebra $\C$ is said to be \emph{right
cocoherent}~\cite[Section~2]{Pmc} if any finitely cogenerated quotient
comodule of a finitely copresented right $\C$\+comodule is finitely
copresented, or equivalently, if any finitely generated subcontramodule
of a finitely presented left $\C$\+contramodule is finitely presented.
 We do \emph{not} know whether any right cocoherent coalgebra needs
to be right quasi-cocoherent.

 A further discussion of quasi-finiteness conditions for coalgebras,
comodules, and contramodules can be found
in~\cite[Sections~5.1\+-5.4]{Psm}.

\medskip
 A $\C$\+comodule is said to be \emph{strongly quasi-finitely
copresented} if it has an injective coresolution consisting of
quasi-finitely cogenerated injective $\C$\+comodules.
 Similarly one could define ``strongly quasi-finitely presented
contramodules''; and the following two lemmas have their obvious
dual-analogous contramodule versions.

\begin{lem} \label{sqfcp-comodules-lemma}
 Let\/ $0\rarrow\K\rarrow\L\rarrow\M\rarrow0$ be a short exact
sequence of\/ $\C$\+comodules.
 Then whenever two of the three comodules\/ $\K$, $\L$, $\M$ are
strongly quasi-finitely copresented, so is the third one.
\end{lem}

\begin{proof}
 Dual to the proof of~\cite[Lemma~2.2]{Pps}
(see \emph{loc.\ cit.}\ for relevant references). 
\end{proof}

 Abusing terminology, we will say that a bounded below complex of
$\C$\+comodules is \emph{strongly quasi-finitely copresented} if it is
quasi-isomorphic to a bounded below complex of quasi-finitely
cogenerated injective $\C$\+comodules.
 Clearly, the class of all strongly quasi-finitely copresented complexes
is closed under shifts and cones in $\sD^+(\C\comodl)$.

\begin{lem} \label{sqfcp-complexes-lemma}
\textup{(a)} Any bounded below complex of strongly quasi-finitely
copresented\/ $\C$\+comodules is strongly quasi-finitely copresented.
\par
\textup{(b)} Let\/ $\M^\bu$ be a complex of\/ $\C$\+comodules
concentrated in the cohomological degrees~$\ge\nobreak n$, where
$n$~is a fixed integer.
 Then\/ $\M^\bu$ is strongly quasi-finitely copresented if and only if
it is quasi-isomorphic to a complex of strongly quasi-finitely
copresented\/ $\C$\+comodules concentrated in the cohomological
degrees~$\ge\nobreak n$. \par
\textup{(c)} Let\/ $\M^\bu$ be a finite complex of\/ $\C$\+comodules
concentrated in the cohomological degrees $n_1\le m\le n_2$.
 Then\/ $\M^\bu$ is strongly quasi-finitely copresented if and only if
it is quasi-isomorphic to a complex of\/ $\C$\+comodules\/ $\K^\bu$
concentrated in the cohomological degrees $n_1\le m\le n_2$ such that
the\/ $\C$\+comodules\/ $\K^m$ are quasi-finitely cogenerated and
injective for all $n_1\le m\le n_2-1$, while the\/ $\C$\+comodule\/
$\K^{n_2}$ is strongly quasi-finitely copresented.
\end{lem}

\begin{proof}
 Dual to~\cite[Lemma~2.3]{Pps}.
\end{proof}

 Lemmas~\ref{sqfcp-comodules-lemma}\+-\ref{sqfcp-complexes-lemma} are
very similar to the respective results from~\cite[Section~2]{Pps}.
 The following examples of pseudo-derived categories of comodules
and contramodules, intended to illustrate the discussion in
Section~\ref{introd-pseudo-derived} of the introduction, are
different from~\cite[Examples~2.5\+-2.6]{Pps}, though, in that
no finiteness or quasi-finiteness conditions
(of the kind discussed above in this section) play any role in
Examples~\ref{pseudo-coderived-comodule-examples}\+-%
\ref{pseudo-contraderived-contramodule-examples}.

 Given a left $\C$\+comodule $\M$ and a right $\C$\+comodule $\N$,
the $k$\+vector spaces $\Cotor^\C_i(\N,\M)$, \ $i=0$, $-1$, $-2$,~\dots\
are defined as the right derived functors of the left exact functor of
cotensor product $\N\oc_\C\M$, constructed by replacing any one or both
of the comodules $\N$ and $\M$ by its injective coresolution, taking
the cotensor product and computing the cohomology~\cite[Sections~0.2.2
and~1.2.2]{Psemi}, \cite[Section~4.7]{Pkoszul}.

\begin{exs} \label{pseudo-coderived-comodule-examples}
 (1)~Let $\C$ be a coassociative coalgebra and $\sS$ be a class of
right $\C$\+comodules.
 Denote by $\sE\subset\sA=\C\comodl$ the full subcategory formed by
all the left $\C$\+comodules $\E$ such that $\Cotor^\C_i(\S,\E)=0$
for all $\S\in\sS$ and all $i<0$.
 Then the full subcategory $\sE\subset\C\comodl$ is a coresolving
subcategory closed under infinite direct sums.
 Thus the derived category $\sD(\sE)$ of the exact category $\sE$ is
a pseudo-coderived category of the abelian category $\C\comodl$, that
is an intermediate quotient category between the coderived category
$\sD^\co(\C\comodl)$ and the derived category $\sD(\C\comodl)$, as
explained in~\cite[Section~4]{PS2} or~\cite[Section~1]{Pps}.

\smallskip
 (2)~In particular, if $\sS=\varnothing$, then one has $\sE=\C\comodl$.
 On the other hand, if $\sS$ is the class of all right $\C$\+comodules,
then $\sE=\C\comodl_\inj$ is the full subcategory of all injective left
$\C$\+comodules.
 In fact, it suffices to take $\sS$ to be the set of all
finite-dimensional right $\C$\+comodules, or just irreducible right
$\C$\+comodules, to force $\sE=\C\comodl_\inj$
\cite[Lemma~3.1(a)]{Prev}.
 In this case, the derived category $\sD(\sE)=\Hot(\sE)$ of the (split)
exact category $\sE$ is equivalent to the coderived category of left
$\C$\+comodules $\sD^\co(\C\comodl)$
\cite[Theorem~5.4(a) or~5.5(a)]{Psemi}, \cite[Theorem~4.4(c)]{Pkoszul}.
\end{exs}

 Given a left $\C$\+comodule $\M$ and a left $\C$\+contramodule $\fT$,
the $k$\+vector spaces $\Coext_\C^i(\M,\fT)$, \ $i=0$, $-1$,
$-2$,~\dots\ are defined as the left derived functors of the right
exact functor of cohomomorphisms $\Cohom_\C(\M,\fT)$, constructed by
replacing either the comodule argument $\M$ by its injective
coresolution, or the contramodule argument $\fT$ by its projective
resolution, or both, taking the $\Cohom_\C$ and computing
the homology~\cite[Sections~0.2.5 and~3.2.2]{Psemi},
\cite[Section~4.7]{Pkoszul}.

\begin{exs} \label{pseudo-contraderived-contramodule-examples}
 (1)~Let $\D$ be a coassociative coalgebra over~$k$ and $\sS$ be
a class of left $\D$\+comodules.
 Denote by $\sF\subset\sB=\D\contra$ the full subcategory formed by
all the left $\D$\+contramodules $\F$ such that $\Coext_\D^i(\S,\F)=0$
for all $\S\in\sS$ and all $i<0$.
 Then the full subcategory $\sF\subset\D\contra$ is a resolving
subcategory closed under infinite products.
 Thus the derived category $\sD(\sF)$ of the exact category $\sF$ is
a pseudo-contraderived category of the abelian category $\D\contra$,
that is an intermediate quotient category between the contraderived
category $\sD^\ctr(\D\contra)$ and the derived category
$\sD(\D\contra)$, as explained in~\cite[Section~4]{PS2}
or~\cite[Section~1]{Pps}.

 \smallskip
 (2)~In particular, if $\sS=\varnothing$, then one has $\sF=\D\contra$.
 On the other hand, if $\sS$ is the class of all left $\D$\+comodules,
then $\sF=\D\contra_\proj$ is the full subcategory of all projective
left $\D$\+contramodules~\cite[Lemma~3.1(b)]{Prev}.
 In fact, it suffices to take $\S$ to be the set of all
finite-dimensional left $\D$\+comodules, or just irreducible left
$\D$\+comodules, to force $\sF=\D\contra_\proj$ \cite[Lemma~A.3]{Psemi}.
 In this case, the derived category $\sD(\sF)=\Hot(\sF)$ of the (split)
exact category $\sF$ is equivalent to the contraderived category of
left $\D$\+contramodules $\sD^\ctr(\D\contra)$ \cite[Theorem~5.4(b)
or~5.5(b)]{Psemi}, \cite[Theorem~4.4(d)]{Pkoszul},
\cite[Corollary~A.6.2]{Pcosh}.
\end{exs}

\Section{Auslander and Bass Classes} \label{auslander-bass-secn}

 We recall the definition of a pseudo-dualizing complex of bicomodules
from Section~\ref{introd-main-results}.
 Let $\C$ and $\D$ be coassociative coalgebras over a field~$k$.

 A \emph{pseudo-dualizing complex} $\L^\bu$ for the coalgebras $\C$ and
$\D$ is a finite complex of $\C$\+$\D$\+bicomodules satisfying
the following two conditions:
\begin{enumerate}
\renewcommand{\theenumi}{\roman{enumi}}
\setcounter{enumi}{1}
\item the complex $\L^\bu$ is strongly quasi-finitely copresented as
a complex of left $\C$\+comodules and as a complex of right
$\D$\+comodules;
\item the homothety maps
$\C^*\rarrow\Hom_{\sD^\b(\comodr\D)}(\L^\bu,\L^\bu[*])$ and
$\D^*{}^\rop\rarrow\Hom_{\sD^\b(\C\comodl)}(\L^\bu,\L^\bu[*])$
are isomorphisms of graded rings. \hbadness=1700
\end{enumerate}
 Here the condition~(ii) refers to the definition of a strongly
quasi-finitely copresented complex of comodules in
Section~\ref{quasi-finite-secn}.
 The complex $\L^\bu$ is viewed as an object of the bounded derived
category of $\C$\+$\D$\+bicomodules $\sD^\b(\C\bicomod\D)$.

 Given a $\C$\+$\D$\+bicomodule $\K$, the functor of contratensor
product $\K\ocn_\D{-}\:\D\contra\allowbreak\rarrow\C\comodl$ is left
adjoint to the functor of comodule homomorphisms $\Hom_\C(\K,{-})\:
\allowbreak\C\comodl\rarrow\D\contra$.
 Hence, in particular, the functor of contratensor product of complexes
$\L^\bu\ocn_\D{-}\:\Hot(\D\contra)\rarrow\Hot(\C\comodl)$ is left
adjoint to the functor $\Hom_\C(\L^\bu,{-})\:\Hot(\C\comodl)\rarrow
\Hot(\D\contra)$.

 We will use the existence theorem of homotopy injective resolutions
of complexes of comodules and homotopy projective resolutions of
complexes of contramodules~\cite[Theorem~2.4]{Pkoszul} in order to
work with the conventional unbounded derived categories of comodules
and contramodules $\sD(\C\comodl)$ and $\sD(\D\contra)$.
 Using the homotopy projective and homotopy injective resolutions of
the second arguments, one constructs the derived functors
$\L^\bu\ocn_\D^\boL{-}\:\sD(\D\contra)\rarrow\sD(\C\comodl)$ and
$\boR\Hom_\C(\L^\bu,{-})\:\sD(\C\comodl)\rarrow\sD(\D\contra)$.
 As a particular case of the general property of the left and right
derived functors (e.~g., in the sense of Deligne~\cite[1.2.1\+-2]{Del})
of left and right adjoint functors, the functor
$\L^\bu\ocn_\D^\boL{-}$ is left adjoint to the functor
$\boR\Hom_\C(\L^\bu,{-})$ \cite[Lemma~8.3]{Psemi}.

 We will use the following simplified notation.
 Given two complexes of left $\C$\+comod\-ules $\M^\bu$ and $\N^\bu$, we
denote by $\Ext_\C^n(\M^\bu,\N^\bu)$ the vector spaces $H^n\boR\Hom_\C
(\M^\bu,\N^\bu)\simeq\Hom_{\sD(\C\comodl)}(\M^\bu,\N^\bu[n])$ of
cohomology of the complex $\boR\Hom_\C(\M^\bu,\N^\bu)=
\Hom_\C(\M^\bu,\J^\bu)$, where $\N^\bu\rarrow\J^\bu$ is
a quasi-isomorphism of complexes of left $\C$\+comodules and $\J^\bu$
is a homotopy injective complex of left $\C$\+comodules.
 Given a complex of right $\D$\+comodules $\N^\bu$ and a complex of
left $\D$\+contramodules $\fT^\bu$, we denote by
$\Ctrtor_n^\D(\N^\bu,\fT^\bu)$ the vector spaces
$H^{-n}(\N^\bu\ocn_\D^\boL\fT^\bu)$ of cohomology of the complex
$\N^\bu\ocn_\D^\boL\fT^\bu=\N^\bu\ocn_\D\P^\bu$, where
$\P^\bu\rarrow\fT^\bu$ is a quasi-isomorphism of complexes of left
$\D$\+contramodules and $\P^\bu$ is a homotopy projective complex of
left $\D$\+contramodules. {\hbadness=2500\par}

 Suppose that the finite complex $\L^\bu$ is situated in
the cohomological degrees $-d_1\le m\le d_2$.
 Then one has $\Ext^n_\C(\L^\bu,\J)=0$ for all $n>d_1$ and all injective
left $\C$\+comodules~$\J$.
 Similarly, one has $\Ctrtor_n^\D(\L^\bu,\P)=0$ for all $n>d_1$ and all
projective left $\D$\+contramodules~$\P$.
 Choose an integer $l_1\ge d_1$ and consider the following full
subcategories in the abelian categories of left $\C$\+comodules and
$\D$\+contramodules:
\begin{itemize}
\item $\sE_{l_1}=\sE_{l_1}(\L^\bu)\subset\C\comodl$ is the full
subcategory consisting of all the $\C$\+comodules $\E$ such that
$\Ext^n_\C(\L^\bu,\E)=0$ for all $n>l_1$ and the adjunction morphism
$\L^\bu\ocn_\D^\boL\boR\Hom_\C(\L^\bu,\E)\rarrow\E$ is an isomorphism
in $\sD^-(\C\comodl)$;
\item $\sF_{l_1}=\sF_{l_1}(\L^\bu)\subset\D\contra$ is the full
subcategory consisting of all the $\D$\+contramodules $\F$ such that
$\Ctrtor_n^\D(\L^\bu,\F)=0$ for all $n>l_1$ and the adjunction morphism
$\F\rarrow\boR\Hom_\C(\L^\bu,\>\L^\bu\ocn_\D^\boL\F)$ is an isomorphism
in $\sD^+(\D\contra)$.
\end{itemize}
 Clearly, for any $l_1''\ge l_1'\ge d_1$, one has $\sE_{l_1'}\subset
\sE_{l_1''}\subset\C\comodl$ and $\sF_{l_1'}\subset\sF_{l_1''}\subset
\D\contra$.
 The category $\sF_{l_1}$ can be called the \emph{Auslander class of\/
contramodules} corresponding to a pseudo-dualizing complex $\L^\bu$,
while the category $\sE_{l_1}$ is the \emph{Bass class of comodules}.

\begin{lem} \label{coalgebra-E-F-closed-kernels-cokernels}
\textup{(a)} The full subcategory\/ $\sE_{l_1}\subset\C\comodl$ is
closed under the cokernels of injective morphisms, extensions, and
direct summands. \par
\textup{(b)} The full subcategory\/ $\sF_{l_1}\subset\D\contra$ is
closed under the kernels of surjective morphisms, extensions, and
direct summands.  \qed
\end{lem}

 The formulation of the next lemma is similar to that
of~\cite[Lemma~3.2]{Pps}, but the proof is quite different.
 Rather, it resembles the related arguments in the proofs
of~\cite[Theorem~3.6]{Pmc} and~\cite[Theorem~4.9]{Pmgm}.

 This lemma shows, in particular, that in the case of a pseudo-dualizing
bicomodule (\,$=$~one-term complex) $\L^\bu=\L$, the pair of adjoint
functors $\Hom_\C(\L,{-})\:\C\comodl\allowbreak\rarrow\D\contra$ and
$\L\ocn_\D{-}\:\D\contra\rarrow\C\comodl$ is a ``left and right
semidualizing adjoint pair'' in the sense of~\cite[Definition~2.1]{GLT}.

\begin{lem} \label{coalgebra-contains-injectives-projectives}
\textup{(a)} The full subcategory\/ $\sE_{l_1}\subset\C\comodl$ contains
all the injective left\/ $\C$\+comodules. \par
\textup{(b)} The full subcategory\/ $\sF_{l_1}\subset\D\contra$ contains
all the projective left\/ $\D$\+contra\-modules.
\end{lem}

\begin{proof}
 Part~(a): we have to check that for any injective left $\C$\+comodule
$\E$ the adjunction morphism $\L^\bu\ocn_\D^\boL\boR\Hom_\C(\L^\bu,\E)=
\L^\bu\ocn_\D^\boL\Hom_\C(\L^\bu,\E)\rarrow\E$ is a quasi-isomorphism.
 It suffices to consider the case of a cofree left $\C$\+comodule
$\E=\C\ot_kV$, where $V$ is a $k$\+vector space.
 Then one has $\Hom_\C(\L^\bu,\E)\simeq\Hom_k(\L^\bu,V)$.

 According to the condition~(ii), there exists a bounded below complex
of quasi-finitely cogenerated injective right $\D$\+comodules $\I^\bu$
endowed with a quasi-isomorphism of complexes of right $\D$\+comodules
$\L^\bu\rarrow\I^\bu$.
 Then we have a quasi-isomorphism of complexes of left
$\D$\+contramodules $\Hom_k(\I^\bu,V)\rarrow\Hom_k(\L^\bu,V)$, and
$\Hom_k(\I^\bu,V)$ is a bounded above complex of projective left
$\D$\+contramodules.
 Hence $\L^\bu\ocn_\D^\boL\Hom_k(\L^\bu,V)=\L^\bu\ocn_\D
\Hom_k(\I^\bu,V)$, and it remains to show that the morphism of
complexes of left $\C$\+comodules
\begin{equation} \label{coalgebra-adjunction-counit-morphism}
 \L^\bu\ocn_\D\Hom_k(\I^\bu,V)\lrarrow\C\ot_kV
\end{equation}
is a quasi-isomorphism.
 The morphism~\eqref{coalgebra-adjunction-counit-morphism} is
constructed in terms of the morphism of complexes of right
$\D$\+comodules $\L^\bu\rarrow\I^\bu$ and the left $\C$\+coaction in
the complex~$\L^\bu$.

 In particular, substituting $V=k$
into~\eqref{coalgebra-adjunction-counit-morphism}, we have a morphism
of complexes of left $\C$\+comodules
\begin{equation} \label{coalgebra-antidualized-c-homothety-map}
 \L^\bu\ocn_\D\I^\bu{}^*\lrarrow\C.
\end{equation}
 Passing to the dual vector spaces
in~\eqref{coalgebra-antidualized-c-homothety-map}, we obtain a map
$\C^*\rarrow\Hom^\D(\I^\bu{}^*,\L^\bu{}^*)$, which is equal to
the composition of the homothety map $\C^*\rarrow\Hom_\D(\L^\bu,\I^\bu)
=\boR\Hom_\D(\L^\bu,\L^\bu)$ with the dualization map
$\Hom_\D(\L^\bu,\I^\bu)\rarrow\Hom^\D(\I^\bu{}^*,\L^\bu{}^*)$.

 As the homothety map is a quasi-isomorphism by the condition~(iii)
and the dualization map is an isomorphism of complexes by
Proposition~\ref{coalgebra-dualization-isomorphism-prop}(b)
(because $\I^\bu$ is a complex of quasi-finitely cogenerated
right $\D$\+comodules, while $\L^\bu$ is a finite complex),
it follows that passing to the dual vector spaces
in~\eqref{coalgebra-antidualized-c-homothety-map} produces
a quasi-isomorphism.
 Hence the map~\eqref{coalgebra-antidualized-c-homothety-map}
is a quasi-isomorphism, too.
 Finally, by
Proposition~\ref{coalgebra-dualization-isomorphism-prop}(c)
the natural map $(\L^\bu\ocn_\D\I^\bu{}^*)\ot_kV\rarrow
\L^\bu\ocn_\D\Hom_k(\I^\bu,V)$ is an isomorphism of complexes.
 Therefore, the map~\eqref{coalgebra-adjunction-counit-morphism}
is also a quasi-isomorphism.

 Part~(b): we have to check that for any projective left
$\D$\+contramodule $\F$ the adjunction morphism
$\F\rarrow\boR\Hom_\C(\L^\bu,\>\L^\bu\ocn_\D^\boL\F)=
\boR\Hom_\C(\L^\bu,\>\L^\bu\ocn_\D\F)$ is a quasi-isomorphism.
 It suffices to consider the case of a free left $\D$\+contramodule
$\F=\Hom_k(\D,V)$, where $V$ is a $k$\+vector space.
 Then one has $\L^\bu\ocn_\D\Hom_k(\D,V)\simeq\L^\bu\ot_kV$.

 According to the condition~(ii), there exists a bounded below complex
of quasi-finitely cogenerated injective left $\C$\+comodules $\J^\bu$
endowed with a quasi-isomorphism of complexes of left $\C$\+comodules
$\L^\bu\rarrow\J^\bu$.
 Then $\boR\Hom_\C(\L^\bu,\>\L^\bu\ot_kV)=\Hom_\C(\L^\bu,\>\J^\bu\ot_kV)$,
and it remains to show that the morphism of complexes of left
$\D$\+contramodules
\begin{equation} \label{coalgebra-adjunction-unit-morphism}
 \Hom_k(\D,V)\lrarrow\Hom_\C(\L^\bu,\>\J^\bu\ot_kV)
\end{equation}
is a quasi-isomorphism.
 The morphism~\eqref{coalgebra-adjunction-unit-morphism} is
constructed in terms of the morphism of complexes of left
$\C$\+comodules $\L^\bu\rarrow\J^\bu$ and the right $\D$\+coaction
in the complex~$\L^\bu$.

 In the same way as in the proof of part~(a), one deduces from
the condition~(iii) using
Proposition~\ref{coalgebra-dualization-isomorphism-prop}(b)
that the natural morphism of complexes of right $\D$\+comodules
\begin{equation} \label{coalgebra-antidualized-d-homothety-map}
 \L^\bu\ocn_{\C^\rop}\J^\bu{}^*\lrarrow\D
\end{equation}
is a quasi-isomorphism.
 Applying the functor $\Hom_k({-},V)$
to~\eqref{coalgebra-antidualized-d-homothety-map}, we see that
the natural map
\begin{equation} \label{coalgebra-adjunction-unit-morphism-rewritten}
\Hom_k(\D,V)\lrarrow\Hom_k(\L^\bu\ocn_{\C^\rop}\J^\bu{}^*,\>V)\.\simeq\.
\Hom^{\C^\rop}(\J^\bu{}^*,\Hom_k(\L^\bu,V))
\end{equation}
is a quasi-isomorphism, too.
 It remains to use 
Proposition~\ref{coalgebra-dualization-isomorphism-prop}(b) again
in order to identify the right-hand sides
of~\eqref{coalgebra-adjunction-unit-morphism}
and~\eqref{coalgebra-adjunction-unit-morphism-rewritten}.
\end{proof}

\begin{rem} \label{coalgebra-closed-sums-products-remark}
 It would be interesting to know whether the analogue
of~\cite[Lemma~3.3]{Pps} holds in our present context, i.~e., whether
the Bass class of comodules $\sE_{l_1}\subset\C\comodl$ is always closed
under infinite direct sums and whether the Auslander class of
contramodules $\sF_{l_1}\subset\D\contra$ is closed under
infinite products.
 A positive answer to these questions would allow to strengthen our
main results in the context of the diagram~\eqref{main-results-diagram}
in Section~\ref{introd-main-results}, as discussed
in Sections~\ref{introd-pseudo-derived}\+-\ref{introd-t-structures}
(cf.~\cite[Sections~0.5 and~0.8]{Pps}).
 Notice that the class of all injective objects in $\C\comodl$ is
always closed under direct sums, and the class of all projective objects
in $\D\contra$ is closed under products~\cite[Section~1.2]{Prev}.
 By the previous lemma, all such injective objects belong to $\sE_{l_1}$,
and all such projective objects belong to $\sF_{l_1}$.
\end{rem}

 In the rest of this section, as well as in the next
Sections~\ref{abstract-classes-secn}\+-\ref{minimal-classes-secn},
we largely follow the exposition in~\cite[Sections~3\+-5]{Pps} with
obvious minimal variations.
 Most proofs are omitted and replaced with references to~\cite{Pps}.

\begin{lem} \label{bass-auslander-complexes}
\textup{(a)} Let\/ $\M^\bu$ be a complex of left\/ $\C$\+comodules
concentrated in the cohomological degrees $-n_1\le m\le n_2$.
 Then\/ $\M^\bu$ is quasi-isomorphic to a complex of left\/
$\C$\+comodules concentrated in the cohomological degrees $-n_1\le m
\le n_2$ with the terms belonging to the full subcategory\/
$\sE_{l_1}\subset\C\comodl$ if and only if\/
$\Ext_\C^n(\L^\bu,\M^\bu)=0$ for $n>n_2+l$ and the adjunction morphism\/
$\L^\bu\ocn_\D^\boL\boR\Hom_\C(\L^\bu,\M^\bu)\rarrow\M^\bu$
is an isomorphism in\/ $\sD^-(\C\comodl)$. \par
\textup{(b)} Let\/ $\fT^\bu$ be a complex of left\/ $\D$\+contramodules
concentrated in the cohomological degrees\/ $-n_1\le m\le n_2$.
 Then\/ $\fT^\bu$ is quasi-isomorphic to a complex of left\/
$\D$\+contramodules concentrated in the cohomological degrees
$-n_1\le m\le n_2$ with the terms belonging to the full subcategory\/
$\sF_{l_1}\subset\D\contra$ if and only if\/
$\Ctrtor^\D_n(\L^\bu,\fT^\bu)=0$ for $n>n_1+l_1$ and the adjunction
morphism\/ $\fT^\bu\rarrow\boR\Hom_\C(\L^\bu,\>
\L^\bu\ocn^\boL_\D\fT^\bu)$ is an isomorphism in\/
$\sD^+(\D\contra)$.
\end{lem}

\begin{proof}
 Similar to~\cite[Lemma~3.4]{Pps}.
\end{proof}

 It follows from Lemma~\ref{bass-auslander-complexes} that the full
subcategory $\sD^\b(\sE_{l_1})\subset\sD(\C\comodl)$ consists of all
the complexes of left $\C$\+comodules $\M^\bu$ with bounded cohomology
such that the complex $\boR\Hom_\C(\L^\bu,\M^\bu)$ also has bounded 
cohomology and the adjunction morphism $\L^\bu\ocn_\D^\boL
\boR\Hom_\C(\L^\bu,\M^\bu)\rarrow\M^\bu$ is an isomorphism.
 Similarly, the full subcategory $\sD^\b(\sF_{l_1})\subset
\sD(\D\contra)$ consists of all the complexes of
left $\D$\+contramodules $\fT^\bu$ with bounded cohomology such that
the complex $\L^\bu\ocn_\D^\boL\fT^\bu$ also has bounded cohomology
and the adjunction morphism $\fT^\bu\rarrow\boR\Hom_\C(\L^\bu,\>
\L^\bu\ocn^\boL_\D\fT^\bu)$ is an isomorphism.

 These two full subcategories can be called the \emph{derived Bass class
of comodules} and the \emph{derived Auslander class of contramodules}.
 General category-theoretic consideration with adjoint
functors~\cite[Theorem~1.1]{FJ} (cf.~\cite[Proposition~2.1]{GLT})
show that the functors
$\boR\Hom_\C(\L^\bu,{-})$ and $\L^\bu\ocn_\D^\boL{-}$ restrict to
a triangulated equivalence between the derived Auslander
and Bass classes,
\begin{equation} \label{bounded-derived-bass-auslander-equi}
 \sD^\b(\sE_{l_1})\simeq\sD^\b(\sF_{l_1}).
\end{equation}

\begin{lem} \label{bass-auslander-taken-into-each-other}
\textup{(a)} For any\/ $\C$\+comodule\/ $\E\in\sE_{l_1}$, the object\/
$\boR\Hom_\C(\L^\bu,\E)\in\sD^\b(\D\contra)$ can be represented by
a complex of\/ $\D$\+contramodules concentrated in the cohomological
degrees $-d_2\le m\le l_1$ with the terms belonging to\/~$\sF_{l_1}$.
\par
\textup{(b)} For any\/ $\D$\+contramodule\/ $\F\in\sF_{l_1}$,
the object\/ $\L^\bu\ocn_\D^\boL\F\in\sD^\b(\C\comodl)$ can be
represented by a complex of $\C$\+comodules concentrated in
the cohomological degrees $-l_1\le m\le d_2$ with the terms belonging
to\/~$\sE_{l_1}$.
\end{lem}

\begin{proof}
 Similar to~\cite[Lemma~3.5]{Pps}.
\end{proof}

 The definitions and discussions of the \emph{coresolution dimension}
of objects of an exact category $\sA$ with respect to its coresolving
subcategory $\sE$ and the \emph{resolution dimension} of objects of
an exact category $\sB$ with respect to its resolving subcategory $\sF$
can be found in~\cite[Section~2]{St} or~\cite[Section~A.5]{Pcosh}
(the terminology in the latter reference is the \emph{right\/
$\sE$\+homological dimension} and \emph{left\/ $\sF$\+homological
dimension}).

\begin{lem} \label{auslander-bass-co-resolution-dimension}
\textup{(a)} For any integers $l_1''\ge l_1'\ge d_1$, the full
subcategory\/ $\sE_{l_1''}\subset\C\comodl$ consists precisely of all
the left\/ $\C$\+comodules whose\/ $\sE_{l_1'}$\+coresolution dimension
does not exceed $l_1''-l_1'$. \par
\textup{(b)} For any integers $l_1''\ge l_1'\ge d_1$, the full
subcategory\/ $\sF_{l_1''}\subset\D\contra$ consists precisely of all
the left\/ $\D$\+contramodules whose\/ $\sF_{l_1'}$\+resolution
dimension does not exceed $l_1''-l_1'$.
\end{lem}

\begin{proof}
 Similar to~\cite[Lemma~3.6]{Pps}.
\end{proof}

\begin{rem}
 It follows from Lemmas~\ref{coalgebra-contains-injectives-projectives}
and~\ref{auslander-bass-co-resolution-dimension} that, for any finite
$n\ge0$, all the left $\C$\+comodules of injective dimension~$\le n$
belong to $\sE_{d_1+n}$ and all the left $\D$\+contramodules of
projective dimension~$\le n$ belong to~$\sF_{d_1+n}$.
\end{rem}

 We refer to~\cite[Section~A.1]{Pcosh} or~\cite[Appendix~A]{Pmgm} for
the definitions of the absolute derived categories appearing in
the next proposition.

\begin{prop} \label{maximal-classes-essentially-the-same}
 For any $l_1''\ge l_1'\ge d_1$ and any conventional or absolute derived
category symbol\/ $\star=\b$, $+$, $-$, $\varnothing$, $\abs+$, $\abs-$,
or\/~$\abs$, the exact inclusion functors of the Bass/Auslander classes
of co/contramodules with varying parameter,
$\sE_{l_1'}\rarrow\sE_{l_1''}$ and\/ $\sF_{l_1'}\rarrow\sF_{l_1''}$,
induce triangulated equivalences
$$
 \sD^\star(\sE_{l_1'})\simeq\sD^\star(\sE_{l_1''})
 \quad\text{and}\quad
 \sD^\star(\sF_{l_1'})\simeq\sD^\star(\sF_{l_1''}).
$$
\end{prop}

\begin{proof}
 Follows from Lemma~\ref{auslander-bass-co-resolution-dimension}
and~\cite[Proposition~A.5.6]{Pcosh}.
\end{proof}

 In particular, the unbounded derived category of the Bass class of
left $\C$\+comodules $\sD(\sE_{l_1})$ is the same for all $l_1\ge d_1$
and the unbounded derived category of the Auslander class of left
$\D$\+contramodules $\sD(\sF_{l_1})$ is the same for all $l_1\ge d_1$.
 We put
$$
 \sD'_{\L^\bu}(\C\comodl)=\sD(\sE_{l_1})
 \quad\text{and}\quad
 \sD''_{\L^\bu}(\D\contra)=\sD(\sF_{l_1}).
$$

 According to the discussion in Section~\ref{introd-t-structures}, it
follows from Lemmas~\ref{coalgebra-E-F-closed-kernels-cokernels}
and~\ref{coalgebra-contains-injectives-projectives} by virtue
of~\cite[Proposition~5.5]{PS2} that the pair of full subcategories
$\sD_\sA^{\le0}(\sE_{l_1})$ and $\sD^{\ge0}(\sE_{l_1})$ forms
a t\+structure of the derived type with the heart $\sA=\C\comodl$ on
the triangulated category $\sD'_{\L^\bu}(\C\comodl)$, while the pair
of full subcategories $\sD^{\le0}(\sF_{l_1})$ and
$\sD^{\ge0}_\sB(\sF_{l_1})$ forms a t\+structure of the derived type
with the heart $\sB=\D\contra$ on the triangulated category
$\sD''_{\L^\bu}(\D\contra)$.
 It is easy to see that these t\+structures do not depend on
the choice of a parameter $l_1\ge d_1$.

 Moreover, following the discussion in
Section~\ref{homotopy-adjusted-coresolving-secn},
the canonical triangulated functor $\sD'_{\L^\bu}(\C\comodl)\rarrow
\sD(\C\comodl)$ induced by the inclusion of exact/abelian categories
$\sE_{l_1}\rarrow\sA$ is a Verdier quotient functor that has
a right adjoint, and similarly,
the canonical triangulated functor $\sD''_{\L^\bu}(\D\contra)\rarrow
\sD(\D\contra)$ induced by the inclusion of exact/abelian categories
$\sF_{l_1}\rarrow\sB$ is a Verdier quotient functor that has
a left adjoint.

 The next theorem, generalizing
the equivalence~\eqref{bounded-derived-bass-auslander-equi}, provides,
in particular, a triangulated equivalence
$$
 \sD'_{\L^\bu}(\C\comodl)=\sD(\sE_{l_1})\simeq
 \sD(\sF_{l_1})=\sD''_{\L^\bu}(\D\contra).
$$
 Thus we obtain a pair of t\+structures of the derived type with
the hearts $\sA$ and $\sB$ on one and the same triangulated category.

\begin{thm} \label{coalgebra-lower-main-thm}
 For any symbol\/ $\star=\b$, $+$, $-$, $\varnothing$, $\abs+$,
$\abs-$, or~$\abs$, there is a triangulated equivalence\/
$\sD^\star(\sE_{l_1})\simeq\sD^\star(\sF_{l_1})$ provided by
(appropriately defined) mutually inverse derived functors\/
$\boR\Hom_\C(\L^\bu,{-})$ and\/ $\L^\bu\ocn_\D^\boL{-}$.
\end{thm}

\begin{proof}
 This is a particular case of
Theorem~\ref{coalgebra-generalized-main-thm} below.
\end{proof}

\Section{Abstract Corresponding Classes}
\label{abstract-classes-secn}

 More generally, suppose that $\sE\subset\C\comodl$ and $\sF\subset
\D\contra$ are two full subcategories satisfying the following
conditions (for some fixed integers $l_1$ and~$l_2$):
\begin{enumerate}
\renewcommand{\theenumi}{\Roman{enumi}}
\item the class of objects $\sE$ is closed under extensions and
the cokernels of injective morphisms in $\C\comodl$, and contains
all the injective left $\C$\+comodules;
\item the class of objects $\sF$ is closed under extensions and
the kernels of surjective morphisms in $\D\contra$, and contains all
the projective left $\D$\+contramodules;
\item for any $\C$\+comodule $\E\in\sE$, the derived category object
$\boR\Hom_\C(\L^\bu,\E)\in\sD^+(\D\contra)$ can be represented by
a complex of $\D$\+contramodules concentrated in the cohomological
degrees $-l_2\le m\le l_1$ with the terms belonging to~$\sF$;
\item for any $\D$\+contramodule $\F\in\sF$, the derived category object
$\L^\bu\ocn_\D^\boL\F\in\sD^-(\C\comodl)$ can be represented by
a complex of $\C$\+comodules concentrated in the cohomological degrees
$-l_1\le m\le l_2$ with the terms belonging to~$\sE$.
\end{enumerate}
 Just as in~\cite[Section~4]{Pps}, one can see from the conditions~(I)
and~(III), or~(II) and~(IV), that $l_1\ge d_1$ and $l_2\ge d_2$ if
$H^{-d_1}(\L^\bu)\ne0\ne H^{d_2}(\L^\bu)$.

 According to Lemmas~\ref{coalgebra-E-F-closed-kernels-cokernels},
\ref{coalgebra-contains-injectives-projectives},
and~\ref{bass-auslander-taken-into-each-other}, the Bass and Auslander
classes $\sE=\sE_{l_1}$ and $\sF=\sF_{l_1}$ satisfy
the conditions~(I\+-IV) with $l_2=d_2$.
 The following lemma provides a kind of converse implication.

\begin{lem} \label{adjunction-isomorphisms-follow}
\textup{(a)} For any\/ $\C$\+comodule\/ $\E\in\sE$, the adjunction
morphism\/ $\L^\bu\ocn_\D^\boL\boR\Hom_\C(\L^\bu,\E)\rarrow\E$ is
an isomorphism in\/ $\sD^\b(\C\comodl)$. \par
\textup{(b)} For any\/ $\D$\+contramodule\/ $\F\in\sF$, the adjunction
morphism\/ $\F\rarrow\boR\Hom_\C(\L^\bu,\>\allowbreak
\L^\bu\ocn_\D^\boL\F)$ is an isomorphism in\/ $\sD^\b(\D\contra)$.
\end{lem}

\begin{proof}
 This can be proved directly in the way similar to the proof
of~\cite[Lemma~4.1]{Pps}, or obtained as a byproduct of the proof of
Theorem~\ref{coalgebra-generalized-main-thm} below.
 (In any case, the argument is based on
Lemma~\ref{coalgebra-contains-injectives-projectives}.)
\end{proof}

 Assuming that $l_1\ge d_1$ and $l_2\ge d_2$, it is clear from
the conditions~(III\+-IV) and Lemma~\ref{adjunction-isomorphisms-follow}
that the inclusions $\sE\subset\sE_{l_1}$ and $\sF\subset\sF_{l_1}$
hold for any two classes of objects $\sE\subset\C\comodl$ and
$\sF\subset\D\contra$ satisfying~(I\+-IV).
 Furthermore, it follows from the conditions~(I\+-II) that
the triangulated functors $\sD^\b(\sE)\rarrow\sD^\b(\C\comodl)$ and
$\sD^\b(\sF)\rarrow\sD^\b(\D\contra)$ induced by the exact inclusions
$\sE\rarrow\C\comodl$ and $\sF\rarrow\D\contra$ are fully faithful.
 Therefore, the triangulated functors $\sD^\b(\sE)\rarrow
\sD^\b(\sE_{l_1})$ and $\sD^\b(\sF)\rarrow\sD^\b(\sF_{l_1})$ are
fully faithful, too.
 Using again the conditions~(III\+-IV), we can conclude that
the equivalence~\eqref{bounded-derived-bass-auslander-equi} restricts
to a triangulated equivalence
\begin{equation} \label{bounded-derived-abstract-equi}
 \sD^\b(\sE)\simeq\sD^\b(\sF).
\end{equation}

 The following theorem is our main result.
 
\begin{thm} \label{coalgebra-generalized-main-thm}
 Let\/ $\sE\subset\C\comodl$ and\/ $\sF\subset\D\contra$ be a pair of
full subcategories of comodules and contramodules satisfying
the conditions~(I\+-IV) for a pseudo-dualizing complex of\/
$\C$\+$\D$\+bicomodules~$\L^\bu$.
 Then for any conventional or absolute derived category symbol\/
$\star=\b$, $+$, $-$, $\varnothing$, $\abs+$, $\abs-$, or\/~$\abs$,
there is a triangulated equivalence $\sD^\star(\sE)\simeq\sD^\star(\sF)$
provided by (appropriately defined) mutually inverse derived functors\/
$\boR\Hom_\C(\L^\bu,{-})$ and\/ $\L^\bu\ocn_\D^\boL\nobreak{-}$.
\end{thm}

\begin{proof}
 The words ``appropriately defined'' here mean ``defined or constructed
using the technology of~\cite[Appendix~A]{Pps}''.
 In the context of the latter, we set
\begin{align*}
 \sA=\C\comodl\,&\supset\,\sE\,\supset\,\sJ=\C\comodl_\inj \\
 \sB=\D\contra\,&\supset\,\sF\,\supset\,\sP=\D\contra_\proj.
\end{align*}
 Consider the adjoint pair of DG\+functors
\begin{alignat*}{2}
 \Psi=\Hom_\C(\L^\bu,{-})\:&\sC^+(\sJ)&&\lrarrow\sC^+(\sB) \\
 \Phi=\L^\bu\ocn_\D{-}\:&\sC^-(\sP)&&\lrarrow\sC^-(\sA).
\end{alignat*}
 Then the constructions of~\cite[Sections~A.2\+-A.3]{Pps} provide
the derived functors $\boR\Psi\:\sD^\star(\sE)\rarrow\sD^\star(\sF)$
and $\boL\Phi\:\sD^\star(\sF)\rarrow\sD^\star(\sE)$.
 The arguments in~\cite[Section~A.4]{Pps} show that the functor
$\boL\Phi$ is left adjoint to the functor $\boR\Psi$, and the first
assertion of~\cite[Theorem~A.7]{Pps} allows to deduce the claim
that they are mutually inverse triangulated equivalences from
the particular case of $\star=\b$, which is the triangulated
equivalence~\eqref{bounded-derived-abstract-equi}.

 Alternatively, applying the second assertion
of~\cite[Theorem~A.7]{Pps} together with
Lemma~\ref{coalgebra-contains-injectives-projectives} allows to
reprove the triangulated
equivalence~\eqref{bounded-derived-abstract-equi} instead of using it,
thus obtaining a proof of Lemma~\ref{adjunction-isomorphisms-follow}.
 We refer to~\cite[proof of Theorem~4.2]{Pps} for the more informal
discussion and further details.
\end{proof}

 According to the discussion in Section~\ref{introd-t-structures}, it
follows from the conditions~(I\+-II) by virtue
of~\cite[Proposition~5.5]{PS2} that the pair of full subcategories
$\sD_\sA^{\le0}(\sE)$ and $\sD^{\ge0}(\sE)$ forms
a t\+structure of the derived type with the heart $\sA=\C\comodl$ on
the triangulated category $\sD(\sE)$, while the pair of
full subcategories $\sD^{\le0}(\sF)$ and $\sD^{\ge0}_\sB(\sF)$
forms a t\+structure of the derived type with the heart $\sB=\D\contra$
on the triangulated category $\sD(\sF)$.
 By Theorem~\ref{coalgebra-generalized-main-thm}, there is
a triangulated equivalence
\begin{equation} \label{unbounded-derived-abstract-equi}
 \sD(\sE)\simeq\sD(\sF).
\end{equation}
 Thus, as in Section~\ref{auslander-bass-secn}, we have a pair of
t\+structures of the derived type with the hearts $\C\comodl$ and
$\D\contra$ on one and the same triangulated
category~\eqref{unbounded-derived-abstract-equi}.

 Moreover, following the discussion in
Section~\ref{homotopy-adjusted-coresolving-secn},
the triangulated functor $\sD(\sE)\rarrow\sD(\C\comodl)$ induced by
the inclusion of exact/abelian categories $\sE\rarrow\C\comodl$ is
a Verdier quotient functor with a right adjoint, while
the triangulated functor $\sD(\sF)\rarrow\sD(\D\contra)$ induced by
the inclusion of exact/abelian categories $\sF\rarrow\D\contra$ is
a Verdier quotient functor with a left adjoint.

 Now suppose that we have two pairs of full subcategories
$\sE_\prime\subset\sE'\subset\C\comodl$ and $\sF_{\prime\prime}
\subset\sF''\subset\D\contra$ such that both the pairs
$(\sE_\prime,\sF_{\prime\prime})$ and $(\sE',\sF'')$ satisfy
the conditions~(I\+-IV).
 Then for any symbol $\star=\b$, $+$, $-$, $\varnothing$, $\abs+$,
$\abs-$, or~$\abs$ there is a commutative diagram of triangulated
functors and triangulated equivalences
\begin{equation} \label{abstract-star-triangulated-diagram}
\begin{tikzcd}
\sD^\star(\sE_\prime)\arrow[rr, Leftrightarrow, no head, no tail]
\arrow[dd] && \sD^\star(\sF_{\prime\prime}) \arrow[dd] \\ \\
\sD^\star(\sE') \arrow[rr, Leftrightarrow, no head, no tail]
&& \sD^\star(\sF'')
\end{tikzcd}
\end{equation}
 The vertical functors are induced by the exact inclusions of
exact categories $\sE_\prime\rarrow\sE'$ and
$\sF_{\prime\prime}\rarrow\sF''$.
 When $\star=\varnothing$, the vertical functors on
the diagram~\eqref{abstract-star-triangulated-diagram} are t\+exact
with respect to both the respective t\+structures and induce
the identity equivalences $\sA\overset=\rarrow\sA$ and
$\sB\overset=\rarrow\sB$ on the hearts.

\begin{rem} \label{pairs-of-t-structures-nontriviality-remark}
 Can one make a \emph{concept} out of the above construction of a pair
of t\+structures (and the similar construction of the pair of
$\infty$\+(co)tilting t\+structures in~\cite[Section~5]{PS2})?
 Let $\sA$ and $\sB$ be abelian categories.
 One can define a \emph{t\+derived pseudo-equivalence} between $\sA$
and $\sB$ as a triangulated category $\sD$ endowed with a pair of
(possibly degenerate) t\+structures of the derived type
$({}'\sD^{\le0},{}'\sD^{\ge0})$
and $({}''\sD^{\le0},{}''\sD^{\ge0})$ whose hearts are
${}'\sD^{\le0}\cap{}'\sD^{\ge0}\simeq\sA$ and
${}''\sD^{\le0}\cap{}''\sD^{\ge0}\simeq\sB$.

 With such a definition, though, any pair of abelian categories is
connected by a trivial t\+derived pseudo-equivalence.
 Set $\sD$ to be the Cartesian product $\sD(\sA)\times\sD(\sB)$.
 Then the pair of full subcategories ${}'\sD^{\le0}=
\sD^{\le0}(\sA)\times\sD(\sB)$ and ${}'\sD^{\ge0}=\sD^{\ge0}(\sA)
\times0\subset\sD$ is a degenerate t\+structure of the derived
type on $\sD$ with the heart $\sA$, while the pair of full
subcategories ${}''\sD^{\le0}=0\times\sD^{\le0}(\sB)$ and
${}''\sD^{\ge0}=\sD(\sA)\times\sD^{\ge0}(\sB)\subset\sD$ is
a degenerate t\+structure of the derived type on $\sD$ with
the heart~$\sB$.

 Perhaps imposing an additional nontriviality condition would make
the notion of a t\+derived pseudo-equivalence more meaningful.
 An obvious idea is to demand existence of a pair of integers~$l_1$
and~$l_2$ such that for any object $X\in\sD$ there exist
distinguished triangles $Y'\rarrow X\rarrow Z''\rarrow Y'[1]$
and $Y''\rarrow X\rarrow Z'\rarrow Y''[1]$ with
$Y'\in{}'\sD^{\ge0}$, \,$Z''\in{}''\sD^{\le l_1-1}$,
\,$Y''\in{}''\sD^{\ge0}$, and $Z'\in{}'\sD^{\le l_2-1}$.
 This condition holds for the pair of t\+structures
$({}'\sD^{\le0},{}'\sD^{\ge0})=(\sD_\sA^{\le0}(\sE),\sD^{\ge0}(\sE))$
and $({}''\sD^{\le0},{}''\sD^{\ge0})=
(\sD^{\le0}(\sF),\sD^{\ge0}_\sB(\sF))$ constructed in this section,
and does \emph{not} hold for the trivial pair of t\+structures
from the previous paragraph.
\end{rem}

\Section{Minimal Corresponding Classes}
\label{minimal-classes-secn}

 Let $\C$ and $\D$ be coassociative coalgebras over a field~$k$, and
let $\L^\bu$ be a pseudo-dualizing complex of $\C$\+$\D$\+bicomodules
concentrated in the cohomological degrees $-d_1\le m\le d_2$.

\begin{prop} \label{minimal-corresponding-classes-prop}
 Fix $l_1=d_1$ and $l_2\ge d_2$.
 Then there exists a unique minimal pair of full subcategories\/
$\sE^{l_2}=\sE^{l_2}(\L^\bu)\subset\C\comodl$ and\/
$\sF^{l_2}=\sF^{l_2}(\L^\bu)\subset\D\contra$ satisfying
the conditions~(I\+-IV).
 For any pair of full subcategories\/ $\sE\subset\C\comodl$ and\/
$\sF\subset\D\contra$ satisfying the conditions~(I\+-IV) one has\/
$\sE^{l_2}\subset\sE$ and\/ $\sF^{l_2}\subset\sF$.
\end{prop}

\begin{proof}
 The full subcategories $\sE^{l_2}\subset\C\comodl$ and $\sF^{l_2}
\subset\D\contra$ are constructed simultaneously by a generation
process similar to the one in~\cite[proof of Proposition~5.1]{Pps}.
 The difference is that, unlike in~\cite{Pps}, we do not require
the classes $\sE^{l_2}$ and $\sF^{l_2}$ to be closed under infinite
direct sums and products, and accordingly do not do the direct
sum/product closure in their construction.
 (Indeed, the direct sum/product closure would be problematic,
as we do not know whether the Bass and Auslander classes $\sE_{l_1}$
and $\sF_{l_1}$ are closed under direct sums/products; see
Remark~\ref{coalgebra-closed-sums-products-remark}.)
 Accordingly, it suffices to repeat the iterative process over the poset
of nonnegative integers and no transfinite iterations are needed.
\end{proof}

\begin{rem}
 Moreover, for any two integers $l_1\ge d_1$ and $l_2\ge d_2$ and any
two full subcategories $\sE\subset\C\comodl$ and $\sF\subset\D\contra$
satisfying the conditions~(I\+-IV) with the parameters~$l_1$ and~$l_2$,
one has $\sE^{l_2}\subset\sE$ and $\sF^{l_2}\subset\sF$.
 This is easily provable by induction with respect to the iterative
construction of the categories $\sE^{l_2}$ and $\sF^{l_2}$ in
Proposition~\ref{minimal-corresponding-classes-prop}
(cf.~\cite[Remark~5.2]{Pps}).
\end{rem}

 One observes that the conditions~(III\+-IV) become weaker as
the parameter~$l_2$ increases.
 Therefore, $\sE^{l_2}\supset\sE^{l_2+1}$ and $\sF^{l_2}\supset
\sF^{l_2+1}$ for all $l_2\ge d_2$.

\begin{lem} \label{abstract-classes-co-resolution-dimension}
 Let $n\ge0$ and $l_1\ge d_1$, \,$l_2\ge d_2+n$ be some integers, and
let\/ $\sE\subset\C\comodl$ and\/ $\sF\subset\D\contra$ be a pair of
full subcategories satisfying the conditions~(I\+-IV) with
the parameters~$l_1$ and~$l_2$.
 Denote by\/ $\sE(n)\subset\C\comodl$ the full subcategory of all left\/
$\C$\+comodules of\/ $\sE$\+coresolution dimension~$\le n$ and by\/
$\sF(n)\subset\D\contra$ the full subcategory of all left\/
$\D$\+contramodules of\/ $\sF$\+resolution dimension~$\le n$.
 Then the pair of classes of comodules and contramodules\/ $\sE(n)$
and\/ $\sF(n)$ satisfies the conditions~(I\+-IV) with the parameters
$l_1+n$ and $l_2-n$.
\end{lem}

\begin{proof}
 Similar to~\cite[Lemma~5.3]{Pps}.
\end{proof}

\begin{prop}
 For any $l_2''\ge l_2'\ge d_2$ and any conventional or absolute derived
category symbol\/ $\star=\b$, $+$, $-$, $\varnothing$, $\abs+$, $\abs-$,
or\/~$\abs$, the exact inclusions\/ $\sE^{l_2''}\rarrow\sE^{l_2'}$ and\/
$\sF^{l_2''}\rarrow\sF^{l_2'}$ induce triangulated equivalences
$$
 \sD^\star(\sE^{l_2''})\simeq\sD^\star(\sE^{l_2'})
 \quad\text{and}\quad
 \sD^\star(\sF^{l_2''})\simeq\sD^\star(\sE^{l_2'}).
$$
\end{prop}

\begin{proof}
 Similar to~\cite[Proposition~5.4]{Pps}.
 Using Lemma~\ref{abstract-classes-co-resolution-dimension}, one checks
that the $\sE^{l_2''}$\+coresolution dimension of any object of
$\sE^{l_2'}$ does not exceed $l_2''-l_2'$ and
the $\sF^{l_2''}$\+resolution dimension of any object of
$\sF^{l_2'}$ does not exceed $l_2''-l_2'$.
 Then one applies~\cite[Proposition~A.5.6]{Pcosh}, as in the proof
of Proposition~\ref{maximal-classes-essentially-the-same}.
\end{proof}

 In particular, the unbounded derived category $\sD(\sE^{l_2})$ of
the minimal corresponding class of left $\C$\+comodules $\sE^{l_2}$
is the same for all $l_2\ge d_2$ and the unbounded derived category
$\sD(\sF^{l_2})$ of the minimal corresponding class of left
$\D$\+contramodules $\sF^{l_2}$ is the same for all $l_2\ge d_2$.
 We put
$$
 \sD^{\L^\bu}_{\prime}(\C\comodl)=\sD(\sE^{l_2})
 \quad\text{and}\quad
 \sD^{\L^\bu}_{\prime\prime}(\D\contra)=\sD(\sF^{l_2}).
$$

 As a particular case of the discussion in
Section~\ref{abstract-classes-secn}, the pair of full subcategories
$\sD_\sA^{\le0}(\sE^{l_2})$ and $\sD^{\ge0}(\sE^{l_2})$ forms
a t\+structure of the derived type with the heart $\sA=\C\comodl$ on
the triangulated category $\sD^{\L^\bu}_{\prime}(\C\comodl)$, while
the pair of full subcategories $\sD^{\le0}(\sF^{l_2})$ and
$\sD_\sB^{\ge0}(\sF^{l_2})$ forms a t\+structure of the derived type
with the heart $\sB=\D\contra$ on the triangulated category
$\sD^{\L^\bu}_{\prime\prime}(\D\contra)$.
 It is easy to see that these t\+structures do not depend on the choice
of a parameter $l_2\ge d_2$.

 Moreover, there is a canonical Verdier quotient functor
$\sD^{\L^\bu}_{\prime}(\C\comodl)\rarrow\sD(\C\comodl)$ that has
a right adjoint, and similarly, there is a canonical Verdier quotient
functor $\sD^{\L^\bu}_{\prime\prime}(\D\contra)\rarrow\sD(\D\contra)$
that has a left adjoint.

 The next theorem provides, in particular, a triangulated equivalence
$$
 \sD^{\L^\bu}_{\prime}(\C\comodl)=\sD(\sE^{l_2})\simeq
 \sD(\sF^{l_2})=\sD^{\L^\bu}_{\prime\prime}(\D\contra).
$$
 Thus, once again, we obtain a pair of t\+structures of the derived
type with the hearts $\sA$ and $\sB$ on one and the same triangulated
category.

\begin{thm} \label{coalgebra-upper-main-thm}
 For any symbol\/ $\star=\b$, $+$, $-$, $\varnothing$, $\abs+$, $\abs-$,
or~$\abs$, there is a triangulated equivalence $\sD^\star(\sE^{l_2})
\simeq\sD^\star(\sF^{l_2})$ provided by (appropriately defined)
mutually inverse derived functors\/ $\boR\Hom_\C(\L^\bu,{-})$ and\/
$\L^\bu\ocn_\D^\boL{-}$.
\end{thm}

\begin{proof}
 This is another particular case of
Theorem~\ref{coalgebra-generalized-main-thm}.
\end{proof}

\Section{Dualizing Complexes}

 Let $\C$ and $\D$ be coassociative coalgebras over~$k$.
 A \emph{dualizing complex} of $\C$\+$\D$\+bicomodules $\L^\bu=\K^\bu$
is a pseudo-dualizing complex (according to the definition in
Section~\ref{auslander-bass-secn}) satisfying the following additional
condition:
\begin{enumerate}
\renewcommand{\theenumi}{\roman{enumi}}
\item As a complex of left $\C$\+comodules, $\K^\bu$ is quasi-isomorphic
to a finite complex of injective $\C$\+comodules, and as a complex of
right $\D$\+comodules, $\K^\bu$ is quasi-isomorphic to a finite complex
of injective $\D$\+comodules.
\end{enumerate}
 In view of Lemma~\ref{sqfcp-complexes-lemma}, the conditions~(i)
and~(ii) taken together can be equivalently restated by saying that,
as a complex of left $\C$\+comodules, $\K^\bu$ is quasi-isomorphic to
a finite complex of quasi-finitely cogenerated injective
$\C$\+comodules, and as a complex of right $\D$\+comodules, $\K^\bu$ is
quasi-isomorphic to a finite complex of quasi-finitely cogenerated
injective $\D$\+comodules.

 Let us choose the parameter $l_2$ in such a way that $\K^\bu$ is
quasi-isomorphic to a complex of (quasi-finitely cogenerated) injective
left $\C$\+comodules concentrated in the cohomological degrees
$-d_1\le m\le l_2$ and to a complex of (quasi-finitely cogenerated)
injective right $\D$\+comodules concentrated in the cohomological
degrees $-d_1\le m\le l_2$.

\begin{prop} \label{dualizing-minimal-classes}
 Let\/ $\L^\bu=\K^\bu$ be a dualizing complex of\/
$\C$\+$\D$\+bicomodules, and let the parameter $l_2$ be chosen as
stated above.
 Then the related minimal corresponding classes\/ $\sE^{l_2}=
\sE^{l_2}(\K^\bu)$ and\/ $\sF^{l_2}=\sF^{l_2}(\K^\bu)$ coincide with
the classes of injective left\/ $\C$\+comodules and projective left\/
$\D$\+contramodules, $\sE^{l_2}=\C\comodl_\inj$ and\/
$\sF^{l_2}=\D\contra_\proj$.
\end{prop}

\begin{proof}
 It suffices to check that the conditions (I\+-IV) hold for
$\sE=\C\comodl_\inj$ and $\sF=\D\contra_\proj$.
 Indeed, the conditions~(I\+-II) are obvious in this case.
 To check~(III), one can assume that $\E$ is a cofree left
$\C$\+comodule, $\E=\C\ot_kV$, where $V$ is a $k$\+vector space.
 Then $\boR\Hom_\C(\K^\bu,\E)=\Hom_\C(\K^\bu,\E)\simeq\Hom_k(\K^\bu,V)$
as a complex of left $\D$\+contramodules.
 Choose a complex of injective right $\D$\+comodules ${}'\J^\bu$
quasi-isomorphic to $\K^\bu$ and concentrated in the cohomological
degrees $-d_1\le m\le l_2$.
 Then the derived category object $\boR\Hom_\C(\K^\bu,\E)\in
\sD^+(\D\contra)$ is represented by the complex of projective
left $\D$\+contramodules $\Hom_k({}'\J^\bu,V)$ concentrated in
the cohomological degrees $-l_2\le m\le d_1$.
 Similarly, to check~(IV), one can assume that $\F$ is a free left
$\D$\+contramodule, $\F=\Hom_k(\D,V)$.
 Then $\K^\bu\ocn_\D^\boL\F=\K^\bu\ocn_\D\F\simeq\K^\bu\ot_kV$ as
a complex of left $\C$\+comodules.
 Choose a complex of injective left $\C$\+comodules ${}''\J^\bu$
quasi-isomorphic to $\K^\bu$ and concentrated in the cohomological
degrees $-d_1\le m\le l_2$.
 Then the derived category object $\K^\bu\ocn_\D^\boL\F\in
\sD^-(\C\comodl)$ is represented by the complex of injective left
$\C$\+comodules ${}''\J^\bu\ot_kV$ concentrated in the cohomological
degrees $-d_1\le m\le l_2$.
\end{proof}

 It is clear from Proposition~\ref{dualizing-minimal-classes}, in view
of the triangulated equivalences $\Hot(\C\comodl_\inj)\simeq
\sD^\co(\C\comodl)$ and $\Hot(\D\contra_\proj)\simeq
\sD^\ctr(\D\contra)$ \cite[Section~4.4]{Pkoszul} mentioned in
the formula~\eqref{derived-co-contra} in
Section~\ref{introd-dedualizing-bicomod}, that for a dualizing
complex of $\C$\+$\D$\+bicomodules $\K^\bu$ one has {\hbadness=1250
$$
 \sD^{\K^\bu}_{\prime}(\C\comodl)=\sD^\co(\C\comodl)
 \quad\text{and}\quad
 \sD^{\K^\bu}_{\prime\prime}(\D\contra)=\sD^\ctr(\D\contra).
$$

}\begin{cor} \label{coalgebra-dualizing-main-cor}
 Let\/ $\C$ and\/ $\D$ be coassociative coalgebras over~$k$, and\/
$\K^\bu$ be a dualizing complex of\/ $\C$\+$\D$\+bicomodules.
 Then there is a triangulated equivalence\/ $\sD^\co(\C\comodl)\simeq
\sD^\ctr(\D\contra)$ provided by the mutually inverse derived functors\/
$\boR\Hom_\C(\K^\bu,{-})$ and\/ $\K^\bu\ocn_\D^\boL{-}$.

 Furthermore, there is a commutative diagram of triangulated
equivalences
$$
\begin{tikzcd}
\sD^\co(\C\comodl)\arrow[rrrrdddd, bend right=7,
"{\boR\Hom_\C(\K^\bu,{-})}"']
\arrow[dddd, Leftrightarrow, no head, no tail] &&&&
\sD^\co(\D\comodl)\arrow[llll,"{\K^\bu\suboc_\D^\boR{-}}"']
\arrow[dddd, Leftrightarrow, no head, no tail] \\ \\ \\ \\
\sD^\ctr(\C\contra)\arrow[rrrr, "{\boL\Cohom_\C(\K^\bu,{-})}"'] &&&&
\sD^\ctr(\D\contra)\arrow[lllluuuu, bend right=7,
"{\K^\bu\ocn_\D^\boL{-}}"']
\end{tikzcd}
$$
where the vertical double lines denote the derived comodule-contramodule
correspondence equivalences~\eqref{derived-co-contra}, and
the horizontal arrows are the derived functors of cotensor product\/
$\oc_\D$ and cohomomorphisms\/ $\Cohom_\C$.
\end{cor}

\begin{proof}
 The first assertion is a particular case of
Theorem~\ref{coalgebra-upper-main-thm}.
 In fact, since $\sE^{l_2}=\C\comodl_\inj$ and $\sF^{l_2}=
\D\contra_\proj$ are split exact categories, the assertion of
Theorem~\ref{coalgebra-upper-main-thm} for a dualizing complex
$\L^\bu=\K^\bu$ reduces to triangulated equivalences between
the (bounded or unbounded) homotopy categories
\begin{equation} \label{dualizing-underived}
 \Hom_\C(\K^\bu,{-})\:\Hot^\star(\C\comodl_\inj)\simeq
 \Hot^\star(\D\contra_\proj)\,:\!\K^\bu\ocn_\D{-}
\end{equation}
for all symbols $\star=\b$, $+$, $-$, or~$\varnothing$.
 Here the functors $\Hom_\C(\K^\bu,{-})$ and $\K^\bu\ocn_\D{-}$ do
not even need to be derived as, following the construction
of~\cite[Appendix~A]{Pps}, they are simply applied to any complex
of injective $\C$\+comodules or (respectively) any complex of projective
$\D$\+contramodules.
 Then the resulting complex is replaced by a complex of
projective/injective objects isomorphic to it in the absolute derived
category of contra/comodules to obtain the triangulated functors
providing the equivalence~\eqref{dualizing-underived}.
 Identifying $\sD^\co(\C\comodl)$ with $\Hot(\C\comodl_\inj)$ and
$\sD^\ctr(\D\contra)$ with $\Hot(\D\contra_\proj)$, as mentioned above,
we obtain the desired triangulated equivalence between the coderived
and the contraderived category,
\begin{equation} \label{dualizing-co-contra-derived}
 \boR\Hom_\C(\K^\bu,{-})\:\sD^\co(\C\comodl)\simeq
 \sD^\ctr(\D\contra)\,:\!\K^\bu\ocn_\D^\boL{-}.
\end{equation}

 Similarly, the right derived functor
$$
 \K^\bu\oc_\D^\boR{-}\:\sD^\co(\D\comodl)\lrarrow\sD^\co(\C\comodl)
$$
is constructed as the composition
$$
 \sD^\co(\D\comodl)\simeq\Hot(\D\comodl_\inj)
 \xrightarrow{\K^\bu\oc_\D{-}}
 \Hot(\C\comodl)\rarrow\sD^\co(\C\comodl)
$$
of the underived cotensor product functor $\K^\bu\oc_\D{-}\:
\Hot(\D\comodl_\inj)\rarrow\Hot(\C\comodl)$ with the Verdier quotient
functor $\Hot(\C\comodl)\rarrow\sD^\co(\C\comodl)$.
 The left derived functor {\hbadness=2200
$$
 \boL\Cohom_\C(\K^\bu,{-})\:\sD^\ctr(\C\contra)
 \lrarrow\sD^\ctr(\D\contra)
$$
is} constructed as the composition
$$
 \sD^\ctr(\C\contra)\simeq\Hot(\C\contra_\proj)
 \xrightarrow{\Cohom_\C(\K^\bu,{-})}
 \Hot(\D\contra)\rarrow\sD^\ctr(\D\contra)
$$
of the underived cohomomorphism functor $\Cohom_\C(\K^\bu,{-})\:
\Hot(\C\contra_\proj)\rarrow\Hot(\D\contra)$ with the Verdier quotient
functor $\Hot(\D\contra)\rarrow\sD^\ctr(\D\contra)$.

 We recall that the downwards directed leftmost vertical equivalence is
constructed by applying the functor $\Hom_\C(\C,{-})$ to complexes
of injective left $\C$\+comodules, while the upwards directed rightmost
vertical equivalence is constructed by applying the functor
$\D\ocn_\D{-}$ to complexes of projective left $\D$\+contramodules
(see formula~\eqref{underived-co-contra} in
Section~\ref{introd-dedualizing-bicomod}).
 Now the upper triangule commutes due to the natural isomorphism of left
$\C$\+comodules $\K\oc_\D(\D\ocn_\D\P)\simeq(\K\oc_\D\D)\ocn_\D\P=
\K\ocn_\D\P$, which holds for any projective left $\D$\+contramodule
$\P$ and any $\C$\+$\D$\+bicomodule $\K$ \cite[Proposition~3.1.1]{Prev},
\cite[Proposition~5.2.1(a)]{Psemi}.
 Similarly, the lower triangle commutes due to the natural isomorphism
of left $\D$\+contramodules $\Cohom_\C(\K,\Hom_\C(\C,\J))\simeq
\Hom_\C(\C\oc_\C\K,\>\J)=\Hom_\C(\K,\J)$, which holds for any injective
left $\C$\+comodule $\J$ and any $\C$\+$\D$\+bicomodule $\K$
\cite[Proposition~3.1.2]{Prev}, \cite[Proposition~5.2.2(a)]{Psemi}.

 Finally, the horizontal functors are triangulated equivalences, since
so are the vertical and diagonal functors.
\end{proof}

\begin{lem} \label{dualizing-one-sided-inverse}
 Let\/ $\K^\bu$ be a dualizing complex of\/ $\C$\+$\D$\+bicomodules.
 Then \par
\textup{(a)} there exists a finite complex of injective left\/
$\D$\+comodules\/ ${}'\K^\bu{}\spcheck$ such that the finite complex
of left\/ $\C$\+comodules\/ $\K^\bu\oc_\D{}'\K^\bu{}\spcheck$ is
quasi-isomorphic to\/~$\C$; \par
\textup{(b)} there exists a finite complex of injective right\/
$\C$\+comodules\/ ${}''\K^\bu{}\spcheck$ such that the finite complex
of right\/ $\D$\+comodules\/ ${}''\K^\bu{}\spcheck\oc_\C\K^\bu$ is
quasi-isomorphic to\/~$\D$.
\end{lem}

\begin{proof}
 Part~(a): by Proposition~\ref{dualizing-minimal-classes}, the derived
category object $\boR\Hom_\C(\K^\bu,\C)\in\sD^+(\D\contra)$ can be
represented by a complex of projective left $\D$\+contramodules $\P^\bu$
concentrated in the cohomological degrees $-l_2\le m\le d_1$.
 Denote by ${}'\K^\bu{}\spcheck$ the complex of injective left
$\D$\+comodules $\D\ocn_\D\P^\bu$.
 Then we have a natural isomorphism of complexes of left $\C$\+comodules
$\K^\bu\oc_\D{}'\K\bu{}\spcheck\simeq\K^\bu\ocn_\D\P^\bu$
by~\cite[Proposition~3.1.1]{Prev} or~\cite[Proposition~5.2.1(a)]{Psemi},
and a natural quasi-isomorphism of complexes of left $\C$\+comodules
$\K^\bu\ocn_\D\P^\bu\rarrow\C$ by
Lemma~\ref{coalgebra-contains-injectives-projectives}(a)
(see formula~\eqref{coalgebra-antidualized-c-homothety-map}).

 Part~(b): since the definition of a dualizing complex of bicomodules
is symmetric with respect to switching the left and right sides and
the roles of the coalgebras $\C$ and $\D$, one can simply apply part~(a)
to the dualizing complex of $\D^\rop$\+$\C^\rop$\+bicomodules~$\K^\rop$.
\end{proof}

\begin{thm} \label{dualizing-derived-morita-takeuchi}
 Let\/ $\C$ and\/ $\D$ be coassociative coalgebras over~$k$, and\/
$\K^\bu$ be a dualizing complex of\/ $\C$\+$\D$\+bicomodules.
 Then \par
\textup{(a)} for any conventional derived category symbol\/ $\star=\b$,
$+$, $-$, or\/~$\varnothing$, there is a triangulated equivalence\/
$\sD^\star(\C\comodl)\simeq\sD^\star(\D\comodl)$ provided by
(appropriately defined) right derived functor\/ $\K^\bu\oc_\D^\boR{-}\:
\sD^\star(\D\comodl)\rarrow\sD^\star(\C\comodl)$; \par
\textup{(b)} for any conventional derived category symbol\/ $\star=\b$,
$+$, $-$, or\/~$\varnothing$, there is a triangulated equivalence\/
$\sD^\star(\C\contra)\simeq\sD^\star(\D\contra)$ provided by
(appropriately defined) left derived functor\/
$\boL\Cohom_\C(\K^\bu,{-})\:\sD^\star(\C\contra)\rarrow
\sD^\star(\D\contra)$.
\end{thm}

\begin{proof}
 Part~(a): as in Theorem~\ref{coalgebra-generalized-main-thm}, the words
``appropriately defined'' here mean ``constructed as
in~\cite[Appendix~A]{Pps}''.
 In fact, a less powerful technology of~\cite[Appendix~B]{Pmgm} is
already sufficient for our purposes here.
 In the context of either of these references, we put
\begin{gather*}
 \sA=\D\comodl=\sE\,\supset\,\sJ=\C\comodl_\inj \\
 \sB=\C\comodl=\sF.
\end{gather*}
 Consider the DG\+functor
$$
 \Psi=\K^\bu\oc_\D{-}\:\sC^+(\sJ)\lrarrow\sC^+(\sB).
$$
 Then the construction of~\cite[Section~A.2]{Pps} provides the derived
functors
$$
 \boR\Psi\:\sD^\star(\D\comodl)\lrarrow\sD^*(\C\comodl)
$$
for all derived category symbols $\star=\b$, $+$, $-$, $\varnothing$,
$\abs+$, $\abs-$, $\co$, or~$\abs$.

 The key observation here is that, for any left $\D$\+comodule $\N$ and
its injective coresolution $\J^\bu$, the complex of left $\C$\+comodules
$\K^\bu\oc_\D\J^\bu$ has cohomology comodules concentrated in
the cohomological degrees $-d_1\le m\le l_2$.
 Indeed, let ${}'\K^\bu$ be a complex of injective right $\D$\+comodules
quasi-isomorphic to $\K^\bu$ whose terms are concentrated in
the cohomological degrees $-d_1\le m\le l_2$.
 Then the complex $\K^\bu\oc_\D\J^\bu$ is quasi-isomorphic, as a complex
of $k$\+vector spaces, to the complex ${}'\K^\bu\oc_\D\J^\bu$, since
for any injective left $\D$\+comodule $\J$ the functor ${-}\oc_\D\J$ is
exact.
 Now for any injective right $\D$\+comodule ${}'\K$ the complex of
$k$\+vector spaces ${}'\K\oc_\D\J^\bu$ is a coresolution of the vector
space ${}'\K\oc_\D\N$, since the functor ${}'\K\oc_\D{-}$ is exact.
 So the cohomology spaces of the complex ${}'\K^\bu\oc_\D\J^\bu$ are
concentrated in the desired cohomological degrees.
 This makes the construction of a derived functor of finite homological
dimension from~\cite[Appendix~A]{Pps} or~\cite[Appendix~B]{Pmgm}
applicable.

 The derived functors $\boR\Psi$ agree with each other as the derived
category symbol~$\star$ varies.
 In particular, for $\star=\co$ and $\star=\varnothing$ we have
a commutative diagram of triangulated functors
$$
\begin{tikzcd}
\sD^\co(\D\comodl) \arrow[rr, "{\K^\bu\suboc_\D^\boR{-}}"]
\arrow[dd, two heads] &&
\sD^\co(\C\comodl) \arrow[dd, two heads] \\ \\
\sD(\D\comodl) \arrow[rr, "{\K^\bu\suboc_\D^\boR{-}}"] &&
\sD(\C\comodl)
\end{tikzcd}
$$
where the vertical arrows are the canonical Verdier quotient functors.
 By Corollary~\ref{coalgebra-dualizing-main-cor}, the upper horizontal
arrow is a triangulated equivalence, hence it follows that the lower
horizontal arrow is, at worst, a Verdier quotient functor.

 Let us check that the kernel of the functor $\K^\bu\oc_\D^\boR{-}\:
\sD(\D\comodl)\rarrow\sD(\C\comodl)$ vanishes.
 By Lemma~\ref{dualizing-one-sided-inverse}(b), there exists a finite
complex of injective right $\C$\+comodules ${}''\K^\bu{}\spcheck$
such that the cotensor product ${}''\K^\bu{}\spcheck\oc_\C\K^\bu$ is
quasi-isomorphic to $\D$ as a complex of right $\D$\+comodules.
 The cotensor product with the complex ${}''\K^\bu{}\spcheck$ defines
a triangulated functor
\begin{equation} \label{cotensor-with-inverse}
 {}''\K^\bu{}\spcheck\oc_\C{-}\:\sD(\C\comodl)\lrarrow\sD(k\vect)
\end{equation}
taking values in the derived category of $k$\+vector spaces.
 The composition
\begin{equation} \label{composition-with-inverse-cotensor}
 \sD(\D\comodl)\xrightarrow{\K^\bu\suboc_\D^\boR{-}}
 \sD(\C\comodl)\xrightarrow{{}''\K^\bu{}\spcheck\suboc_\C{-}}
 \sD(k\vect)
\end{equation}
is the derived functor of cotensor product with
${}''\K^\bu{}\spcheck\oc_\C\K^\bu$, which is isomorphic to
the forgetful functor $\sD(\D\comodl)\rarrow\sD(k\vect)$.
 It follows that any complex of left $\D$\+comodules annihilated by
the functor $\K^\bu\oc_\D^\boR{-}$ is acyclic.
 We have proved that the functor
$$
 \K^\bu\oc_\D^\boR{-}\:\sD(\D\comodl)\lrarrow\sD(\C\comodl)
$$
is a triangulated equivalence.

 Finally, for $\star=\b$, $+$, or~$-$ we have a commutative diagram of
triangulated functors
$$
\begin{tikzcd}
\sD^\star(\D\comodl) \arrow[rr, "{\K^\bu\suboc_\D^\boR{-}}"]
\arrow[dd, tail] &&
\sD^\star(\C\comodl) \arrow[dd, tail] \\ \\
\sD(\D\comodl) \arrow[rr, "{\K^\bu\suboc_\D^\boR{-}}"] &&
\sD(\C\comodl)
\end{tikzcd}
$$
where the vertical arrows are the canonical fully faithful inclusions.
 We already know that the lower horizontal arrow is an equivalence;
hence the upper horizontal arrow is, at worst, fully faithful.

 Let us check that a complex of left $\D$\+comodules has
$\star$\+bounded cohomology comodules whenever its image under
the functor $\K^\bu\oc_\D^\boR{-}$ has $\star$\+bounded cohomology.
 Indeed, the triangulated functor~\eqref{cotensor-with-inverse}
takes $\star$\+bounded complexes to $\star$\+bounded complexes,
$$
 {}''\K^\bu{}\spcheck\oc_\C{-}\:
 \sD^\star(\C\comodl)\lrarrow\sD^\star(k\vect),
$$
so the desired assertion follows from the fact that the composition
of functors~\eqref{composition-with-inverse-cotensor} is isomorphic
to the forgetful functor $\sD(\D\comodl)\rarrow\sD(k\vect)$.
 The proof of part~(a) is finished.

 Part~(b) is similar.
 In the context of~\cite[Appendix~A]{Pps} or~\cite[Appendix~B]{Pmgm},
we put
\begin{gather*}
 \sB=\C\contra=\sF\,\supset\,\sP=\C\contra_\proj \\
 \sA=\D\contra=\sE.
\end{gather*}
 Consider the DG\+functor
$$
 \Phi=\Cohom_\C(\K^\bu,{-})\:\sC^-(\sP)\lrarrow\sC^-(\sA).
$$
 Then the construction of~\cite[Section~A.3]{Pps} provides the derived
functors
$$
 \boL\Phi\:\sD^\star(\C\contra)\lrarrow\sD^*(\D\contra)
$$
for all derived category symbols $\star=\b$, $+$, $-$, $\varnothing$,
$\abs+$, $\abs-$, $\ctr$, or~$\abs$.
 The key observation here is that, for any left $\C$\+contramodule
$\fT$ and its projective resolution $\P^\bu$, the complex of left
$\D$\+contramodules $\Cohom_\C(\K^\bu,\P^\bu)$ has cohomology
contramodules concentrated in the cohomological degrees
$-l_2\le m\le d_1$.

 The rest of the argument uses
Lemma~\ref{dualizing-one-sided-inverse}(a) in the way similar
to the above.
\end{proof}

\begin{rem}
 The exposition above simplifies considerably when a dualizing complex
$\K^\bu$ is quasi-isomorphic to a (finite) complex of
$\C$\+$\D$\+bicomodules that are \emph{simultaneously} quasi-finitely
cogenerated injective left $\C$\+comodules and quasi-finitely
cogenerated injective right $\D$\+comodules.
 Then, following~\cite[Sections~1.8\+-1.9]{Tak}
(cf.\ Remark~\ref{takeuchi-co-hom-explained} above), both the complexes
${}'\K^\bu{}\spcheck$ and ${}''\K^\bu{}\spcheck$ in
Lemma~\ref{dualizing-one-sided-inverse} can be constructed
as complexes of $\D$\+$\C$\+bicomodules together with a natural
quasi-isomorphism of complexes of $\C$\+$\C$\+bicomodules $\C\rarrow
\K^\bu\oc_\D{}''\K^\bu{}\spcheck$ and a natural quasi-isomorphism of
complexes of $\D$\+$\D$\+bicomodules $\D\rarrow {}'\K^\bu{}\spcheck
\oc_\C\K^\bu$ (implying that ${}'\K^\bu{}\spcheck$ and
${}''\K^\bu{}\spcheck$ are quasi-isomorphic to each other as complexes
of $\D$\+$\C$\+bicomodules).
 Given such ``inverse complex(es) of $\D$\+$\C$\+bicomodules'' to
$\K^\bu$, one easily constructs the derived equivalences in
Theorem~\ref{dualizing-derived-morita-takeuchi} (including even
the case of $\star=\abs+$, $\abs-$, or~$\abs$) similarly to
the underived theory of~\cite{Tak} (cf.~\cite[Sections~2\+-3]{Far};
see also the discussion of Morita morphisms and Morita equivalences
in~\cite[Section~7.5]{Psemi}).
\end{rem}

\Section{Dedualizing Complexes}

 Recall the following definitions from~\cite[Section~3]{Pmc}.
 A finite complex of left $\C$\+comodules $\M^\bu$ is said to have
\emph{projective dimension~$\le l$} if $\Ext_\C^n(\M^\bu,\N)=0$
for all left $\C$\+comodules $\N$ and all the integers $n>l$.
 A finite complex of right $\D$\+comodules $\N^\bu$ is said to
have \emph{contraflat dimension~$\le l$} if $\Ctrtor^\D_n(\N^\bu,\fT)
=0$ for all left $\D$\+contramodules $\fT$ and all the integers $n>l$.
 Here we use the notation $\Ext_\C^*$ and $\Ctrtor^\D_*$ introduced
in the beginning of Section~\ref{auslander-bass-secn}.

 A \emph{dedualizing complex} of $\C$\+$\D$\+bicomodules $\L^\bu=\B^\bu$
is a pseudo-dualizing complex (according to the definition in
Section~\ref{auslander-bass-secn}) satisfying the following additional
condition:
\begin{enumerate}
\renewcommand{\theenumi}{\roman{enumi}}
\item the complex $\B^\bu$ has finite projective dimension as a complex
of left $\C$\+co\-modules and finite contraflat dimension as a complex
of right $\D$\+comodules.
\end{enumerate}
 This is a version of the definition of a dedualizing complex
in~\cite[Section~3]{Pmc}, extended from the case of cocoherent
coalgebras to arbitrary ones and from the case of finitely copresented
comodules to quasi-finitely copresented ones.

 Let us choose the parameter~$l_1$ in such a way that the projective
dimension of the complex of left $\C$\+comodules $\B^\bu$ does not
exceed~$l_1$ and the contraflat dimension of the complex of right
$\D$\+comodules $\B^\bu$ does not exceed~$l_1$.
 One can easily see that any one of these two conditions implies
$l_1\ge d_1$ (take $\N=\C$ or $\fT=\D^*$ in the above definitions of
the projective and contraflat dimensions).

\begin{lem} \label{dedualizing-bass-auslander}
 Let\/ $\B^\bu$ be a dedualizing complex of\/ $\C$\+$\D$\+bicomodules,
and let the parameter~$l_1$ be chosen as stated above.
 Then the related Bass and Auslander classes\/ $\sE_{l_1}=
\sE_{l_1}(\B^\bu)$ and\/ $\sF_{l_1}=\sF_{l_1}(\B^\bu)$ coincide with
the whole categories of left\/ $\C$\+comodules and left\/
$\D$\+contramodules, $\sE_{l_1}=\C\comodl$ and\/
$\sF_{l_1}=\D\contra$.
\end{lem}

\begin{proof}
 In view of Lemma~\ref{adjunction-isomorphisms-follow} and
the subsequent discussion, it suffices to check that
the conditions~(I\+-IV) of Section~\ref{abstract-classes-secn} hold for
the classes $\sE=\C\comodl$ and $\sF=\D\contra$ with the given
parameter~$l_1$ and some $l_2\ge d_2$.
 Indeed, let us take $l_2=d_2$.
 Then the conditions~(I\+-II) are obvious, and the conditions~(III\+-IV)
follow from~(i).
\end{proof}

 It is clear from Lemma~\ref{dedualizing-bass-auslander} that for
a dedualizing complex of $\C$\+$\D$\+bicomodules $\B^\bu$ one has
$$
 \sD'_{\B^\bu}(\C\comodl)=\sD(\C\comodl)
 \quad\text{and}\quad
 \sD''_{\B^\bu}(\D\contra)=\sD(\D\contra).
$$
 The next corollary is a generalization of~\cite[Theorem~3.6]{Pmc}.

\begin{cor}
 Let\/ $\C$ and\/ $\D$ be coassociative coalgebras over~$k$, and\/
$\B^\bu$ be a dedualizing complex of\/ $\C$\+$\D$\+bicomodules.
 Then for any conventional or absolute derived category symbol\/
$\star=\b$, $+$, $-$, $\varnothing$, $\abs+$, $\abs-$, or\/~$\abs$,
there is a triangulated equivalence\/ $\sD^\star(\C\comodl)\simeq
\sD^\star(\D\contra)$ provided by (appropriately defined) mutually
inverse derived functors\/ $\boR\Hom_\C(\B^\bu,{-})$ and\/
$\B^\bu\ocn_\D^\boL{-}$.
\end{cor}

\begin{proof}
 This is a particular case of Theorem~\ref{coalgebra-lower-main-thm}.
\end{proof}

\Section{Main Diagram}

 The aim of this section is to discuss the results
formulated in Section~\ref{introd-main-results}, and in particular,
the ones encoded in the diagram~\eqref{main-results-diagram}.
 In fact, all of them have been proved already in the preceding
sections; see, in particular,
the diagram~\eqref{abstract-star-triangulated-diagram} in
Section~\ref{abstract-classes-secn}, generalizing the middle
square in~\eqref{main-results-diagram}.
 So in this section we present a kind of overview and conclusion.
 
 Let $\L^\bu$ be a pseudo-dualizing complex of bicomodules for a pair
of coalgebras $\C$ and $\D$, and let $\sE\subset\C\comodl$ and
$\sF\subset\D\contra$ be a pair of classes of comodules and
contramodules satisfying the conditions~(I\+-IV) with some parameters
$l_1\ge d_1$ and $l_2\ge d_2$.
 Then there is the following diagram of triangulated functors:
\begin{equation} \label{abstract-main-diagram}
\begin{tikzcd}
\Hot(\C\comodl) \arrow[d, two heads]
\arrow[ddd, two heads, bend right=71] &&&&
\Hot(\D\contra) \arrow[d, two heads]
\arrow[ddd, two heads, bend left=71] \\
\sD^\co(\C\comodl) \arrow[d] 
\arrow[u, tail, bend right=70]
\arrow[dd, two heads, bend right=62] &&&&
\sD^\ctr(\D\contra) \arrow[d] 
\arrow[u, tail, bend left=70]
\arrow[dd, two heads, bend left=62] \\
\sD(\sE) \arrow[d, two heads]
\arrow[rrrr, Leftrightarrow, no head, no tail] &&&&
\sD(\sF) \arrow[d, two heads] \\
\sD(\C\comodl) \arrow[uuu, tail, bend right=82]
\arrow[uu, tail, bend right=70]
\arrow[u, tail, bend right=60] &&&&
\sD(\D\contra) \arrow[uuu, tail, bend left=82]
\arrow[uu, tail, bend left=70]
\arrow[u, tail, bend left=60]
\end{tikzcd}
\end{equation}

 Here the uppermost straight vertical arrows $\Hot(\C\comodl)\rarrow
\sD^\co(\C\comodl)$ and $\Hot(\sD\contra)\rarrow\sD^\ctr(\D\contra)$
are the canonical Verdier quotient functors.
 The functor $\Hot(\C\comodl)\rarrow\sD^\co(\C\comodl)$ has a fully
faithful right adjoint, identifying the coderived category
$\sD^\co(\C\comodl)$ with the homotopy category of complexes of
injective comodules $\Hot(\C\comodl_\inj)$
\cite[Theorem~4.4(a,c)]{Pkoszul}.
 Similarly, the functor $\Hot(\D\contra)\rarrow\sD^\ctr(\D\contra)$
has a fully faithful left adjoint, identifying the contraderived
category $\sD^\ctr(\D\contra)$ with the homotopy category of complexes
of projective contramodules $\Hot(\D\contra_\proj)$
\cite[Theorem~4.4(b,d)]{Pkoszul}.
 These fully faithful adjoints are shown on the diagram as
the uppermost short curvilinear arrows.

 The functor $\sD^\co(\C\comodl)\rarrow\sD(\sE)$ shown by the leftmost
middle straight vertical arrow is constructed as the composition
$$
 \sD^\co(\C\comodl)\simeq\Hot(\C\comodl_\inj)\lrarrow\sD(\sE)
$$
of the natural equivalence $\sD^\co(\C\comodl)\simeq
\Hot(\C\comodl_\inj)$ with the triangulated functor
$\Hot(\C\comodl_\inj)\rarrow\sD(\sE)$ induced by the inclusion of
additive/exact categories $\C\comodl_\inj\rarrow\sE$.

 Similarly, the functor $\sD^\ctr(\D\contra)\rarrow\sD(\sF)$ shown by
the rightmost middle straight vertical arrow is constructed as
the composition
$$
 \sD^\ctr(\D\contra)\simeq\Hot(\D\contra_\proj)\lrarrow\sD(\sF)
$$
of the natural equivalence $\sD^\ctr(\D\contra)\simeq
\Hot(\D\contra_\proj)$ with the triangulated functor
$\Hot(\D\contra_\proj)\rarrow\sD(\sF)$ induced by the inclusion of
additive/exact categories $\D\contra_\proj\rarrow\sF$.

 The horizontal double line is the triangulated equivalence of
Theorem~\ref{coalgebra-generalized-main-thm} for $\star=\varnothing$.

 The left lower straight vertical arrow $\sD(\sE)\rarrow\sD(\C\comodl)$
is induced by the inclusion of exact/abelian categories $\sE\rarrow
\C\comodl$.
 Similarly, the right lower straight vertical arrow $\sD(\sF)\rarrow
\sD(\D\contra)$ is induced by the inclusion of exact/abelian
categories $\sF\rarrow\D\contra$.
 It was shown in Section~\ref{homotopy-adjusted-coresolving-secn}
that these two functors are Verdier quotient functors.
 
 The composition $\sD^\co(\C\comodl)\rarrow\sD(\sE)\rarrow
\sD(\C\comodl)$ is the canonical Verdier quotient functor.
 So is the long composition $\Hot(\C\comodl)\rarrow\sD^\co(\C\comodl)
\rarrow\sD(\C\comodl)$.
 These Verdier quotient functors (shown on the diagram by
the two-headed long curvilinear arrows in the left-hand side) have
fully faithful right adjoints (shown by the long curvilinear arrows with
tails), which were constructed in~\cite[Sections~2.4 and~5.5]{Pkoszul}
(see also~\cite[Theorem~1.1(c)]{Pmc}).

 Similarly, the composition $\sD^\ctr(\D\contra)\rarrow\sD(\sF)\rarrow
\sD(\D\contra)$ is the canonical Verdier quotient functor.
 So is the long composition $\Hot(\D\contra)\rarrow\sD^\ctr(\D\contra)
\rarrow\sD(\D\contra)$.
 These Verdier quotient functors (shown on the diagram by
the two-headed long curvilinear arrows in the right-hand side) have
fully faithful left adjoints (shown by the long curvilinear arrows with
tails), which were constructed in~\cite[Sections~2.4 and~5.5]{Pkoszul}
(see also~\cite[Theorem~1.1(a)]{Pmc}).

 The lower curvilinear arrows with tails show triangulated functors
$\sD(\C\comodl)\rarrow\sD(\sE)$ and $\sD(\D\contra)\rarrow\sD(\sF)$,
which are fully faithful and adjoint to the above-mentioned functors
$\sD(\sE)\rarrow\sD(\C\comodl)$ and $\sD(\sF)\rarrow\sD(\D\contra)$ on
the respective sides.
 These functors are described in the following proposition.

\begin{prop}
\textup{(a)} The triangulated functor\/ $\sD(\sE)\rarrow\sD(\C\comodl)$
induced by the inclusion of exact/abelian categories\/ $\sE\rarrow
\C\comodl$ is a Verdier quotient functor, and has a fully faithful
right adjoint. \par
\textup{(b)} The triangulated functor\/ $\sD(\sF)\rarrow\sD(\D\contra)$
induced by the inclusion of exact/abelian categories\/ $\sF\rarrow
\D\contra$ is a Verdier quotient functor, and has a fully faithful
left adjoint.
\end{prop}

\begin{proof}
 This is a particular case of Theorem~\ref{weak-pseudoderived}, which
was mentioned already in Section~\ref{abstract-classes-secn}.
 In part~(a), the desired functor $\rho\:\sD(\C\comodl)\rarrow\sD(\sE)$
is constructed as the composition
$$
 \sD(\C\comodl)\overset\theta\lrarrow\sD^\co(\C\comodl)\lrarrow\sD(\sE),
$$
or which is the same, the composition
$$
 \sD(\C\comodl)\lrarrow\Hot(\C\comodl)\lrarrow\sD(\sE)
$$
of the functors on the diagram.
 In part~(b), the desired functor $\lambda\:\sD(\D\contra)\rarrow
\sD(\sF)$ is constructed as the composition
$$
 \sD(\D\contra)\overset\varkappa\lrarrow\sD^\ctr(\D\contra)
 \lrarrow\sD(\sF),
$$
or which is the same, the composition
$$
 \sD(\D\contra)\lrarrow\Hot(\D\contra)\lrarrow\sD(\sF)
$$
of the functors on the diagram.
\end{proof}

 Finally, according to Sections~\ref{introd-t-structures}
and~\ref{abstract-classes-secn}, there are two (possibly degenerate)
t\+structures of the derived type on the triangulated category
$\sD(\sE)\simeq\sD(\sF)$.
 The abelian hearts are $\sA=\C\comodl$ and $\sB=\D\contra$.
 This pair of t\+structures satisfies the nontriviality condition of
Remark~\ref{pairs-of-t-structures-nontriviality-remark}.

\end{document}